\documentclass[11pt,a4paper]{article}
\usepackage[utf8]{inputenc}
\usepackage{amsmath,amsfonts,amssymb,amsthm,subcaption}
\usepackage{graphicx}
\graphicspath{ {figures/} }
\usepackage{hyperref}

\author{Rodrigo T. Sato Mart\'in de Almagro\footnote{\href{mailto:rodrigo.t.sato@fau.de}{rodrigo.t.sato@fau.de}} 
	\\[2mm]
	{\small  Friedrich-Alexander-Universit\"at Erlangen-N\"urnberg (FAU),}\\
	{\small  Institute of Applied Dynamics}\\
	{\small Immerwahrstrasse 1, 91058 Erlangen, Germany}}
\title{High-order integrators for Lagrangian systems on homogeneous spaces via nonholonomic mechanics}

\newtheorem{proposition}{Proposition}[section]

\theoremstyle{remark}
	\newtheorem*{remark}{Remark}

\theoremstyle{definition}

\theoremstyle{definition}

\providecommand{\keywords}[1]
{
  \small	
  \textbf{\textit{Keywords---}} #1
}

\begin{document}
\maketitle

\begin{abstract}
	In this paper, high-order numerical integrators on homogeneous spaces will be presented as an application of nonholonomic partitioned Runge-Kutta Munthe-Kaas (RKMK) methods on Lie groups.
	A homogeneous space $M$ is a manifold where a group $G$ acts transitively. Such a space can be understood as a quotient $M \cong G/H$, where $H$ a closed Lie subgroup, is the isotropy group of each point of $M$. The Lie algebra of $G$ may be decomposed into $\mathfrak{g} = \mathfrak{m} \oplus \mathfrak{h}$, where $\mathfrak{h}$ is the subalgebra that generates $H$ and $\mathfrak{m}$ is a subspace. Thus, variational problems on $M$ can be treated as nonholonomically constrained problems on $G$, by requiring variations to remain on $\mathfrak{m}$.
	Nonholonomic partitioned RKMK integrators are derived as a modification of those obtained by a discrete variational principle on Lie groups, and can be interpreted as obeying a discrete Chetaev principle. These integrators tend to preserve several properties of their purely variational counterparts.
\end{abstract}

\keywords{
Homogeneous spaces,
High-order integrators,
Runge-Kutta Munthe-Kaas,
Nonholonomic mechanics,
Lie groups
}

\section{Introduction}
Geometric integrators are numerical methods for continuous-time dynamical systems that preserve the geometric structure of the system. This structure preservation may focus on different aspects of the systems to be integrated, but can be mainly summarized as the preservation of certain submanifolds where the system evolves. Such are the cases of symplectic integrators, of which variational integrators are a subset, and energy-preserving methods.\\

In this work we focus on geometric integrators obtained as modifications of variational integrators, applied to dynamical systems on homogeneous spaces, which are particular kinds of smooth manifolds. The proposed methods intend to preserve some of the characteristics of the variational integrators from which they derive, such as good long-term energy behaviour or preservation of momenta, while ensuring that the evolution remains constrained to the corresponding manifold.\\

Previous works on this topic can be grouped into general purpose methods and methods adapted to mechanical systems, i.e. Lagrangian or Hamiltonian systems which preserve a rich geometric structure in their evolution. Some interesting and broad-reaching articles of the first group are \cite{MuntheKaas_Zanna} and \cite{MuntheKaas_Verdier}. %Another example is \cite{Celledoni_Lie}, where general purpose high-order methods are applied to mechanical systems.
On the second group we have for instance \cite{LeeLeokMcClamroch_Spheres}, \cite{ZenkovLeokBloch} and are mainly restricted to low-order methods (order 2 convergence).\\

This work should be cast into the second group. Our intention is to offer methods that can achieve high order and be regarded as natural extensions of variational methods to homogeneous spaces. In order to do so, we frame the problems in higher-dimensional spaces in a natural manner, and constrain the dynamics to the original homogeneous space. In doing so, we transform the problems into nonholonomic problems, where the resulting constraints are non-integrable. See \cite{Cortes02} or \cite{Bloch03} for modern reviews on the topic.\\

\section{Theoretical setup}
\label{sec:theory}
A homogeneous space, see for instance \cite{Lee03}, is a pair $(G,M)$, where $M$ is a smooth manifold and $G$ is a connected Lie group that acts transitively on $M$. In this work we assume that both $G$ and $M$ are finite-dimensional.\\

Let us denote the left action of the group on $M$ by $\sigma: G \times M \to M$. Transitivity of the action means that for every pair $x,y \in M$, there exists $g \in G$ such that $\sigma(g,x) = y$. Typical examples to keep in mind throughout this text are $(n-1)$-spheres, $(SO(n),S^{n-1})$ with $n \geq 2$. We will denote the action of the group on itself by simple juxtaposition, i.e. if $g, h \in G$, then $g h, h g \in G$.\\

Let $\mathfrak{g}$ denote the Lie algebra of $G$. This is identified with the tangent space at the identity $e \in G$, i.e. $\mathfrak{g} \cong T_e G$. Given $\xi \in \mathfrak{g}$, the action at a point $x \in M$ associates with $\xi$ a tangent vector at that point
\begin{equation*}
\xi_M (x) := \left.\frac{\mathrm{d}}{\mathrm{d}t} \right\vert_{t=0} \sigma(\exp(t \xi),x) = D_1 \sigma_{(e,x)}(\xi)\,.
\end{equation*}
where $\exp: \mathfrak{g} \to G$ denotes the exponential map of the Lie group.

It is well known that a homogeneous space is diffeomorphic to the quotient $G/H$ where $H \subset G$ is the isotropy group of an arbitrary point on $M$, i.e.
\begin{equation*}
H := \left\lbrace h \in G \,\vert\, \sigma(h,x) = x\right\rbrace\,.
\end{equation*}
We denote the quotient map by $\pi: G \to G/H \cong M$. Choosing an \emph{origin}, e.g. $\pi(e) = x_0 \in M$, we can give an explicit expression for this map as $\pi(g) = \sigma(g,x_0)$.\\

Let $x \in C^k(I,M)$, with $k \geq 2$, denote a curve defined on the interval $I \subset \mathbb{R}$. By transitivity, there exist curves $g \in C^k(I,G)$ such that $x = \pi(g)$. However, due to isotropy, given a curve on $M$ there is no unique curve on $G$ associated to it and a choice needs to be made to determine the problem completely. One possible choice is as follows. Let $\mathfrak{h}$ denote the Lie algebra of $H$. Then, it is always possible to decompose the algebra of $G$ as a direct sum of vector spaces (i.e. not necessarily a direct sum of Lie algebras), $\mathfrak{g} = \mathfrak{m} \oplus \mathfrak{h}$. To simplify this process, we assume a non-degenerate symmetric bilinear form on $\mathfrak{g}$, $B: \mathfrak{g} \times \mathfrak{g} \to \mathbb{R}$, and since $\mathfrak{g}$ is a finite dimensional vector space, it is always possible to find one. With it, we can define an isomorphism between the algebra and its dual, $\flat_B: \mathfrak{g} \to \mathfrak{g}^*$. This allows us to define
\begin{equation*}
\mathfrak{m} := \ker \, \flat_B(\mathfrak{h}) = \left\lbrace \xi \in \mathfrak{g} \, | \, B(\xi, \eta) = 0, \forall \eta \in \mathfrak{h} \right\rbrace\,.
\end{equation*}
Clearly, $D_1 \sigma_{(e,x)}(\mathfrak{m}) = T_x M$, for all $x \in M$.

\begin{remark}
Every tangent space of $G$ can be put into correspondence with $\mathfrak{g}$ by left or right translation. Therefore, the direct sum decomposition described above leads to two generally inequivalent distributions on $G$, i.e. subspaces $\mathcal{D}_g \subset T_g G$ for every $g \in G$, depending on whether we propagate via left or right translation.
\end{remark}

\begin{remark}
When $H$ is semisimple its Killing form is non-degenerate and provides a natural choice of symmetric bilinear form and, thus, a natural decomposition which is also bi-invariant, leading to a unique distribution. In addition, when $H$ is connected the space $M$ is said to be a reductive homogeneous space.
\end{remark}

Given that we can define a curve on $G$ from another curve on the algebra, i.e.
\begin{equation*}
\dot{g}(t) = T_e \mathcal{L}_{g(t)} \mathrm{H}(t) = T_e \mathcal{R}_{g(t)} \eta(t), \quad \forall t \in I
\end{equation*}
with $\eta, \mathrm{H}\footnote{The lower and uppercase Greek letters `eta' have been chosen in order to differentiate it from $\xi$, $\Xi \in \mathfrak{g}$ used in the discrete case, with a completely different meaning.} \in C^k(I,\mathfrak{g})$, we may restrict these curves on the algebra to the subspace $\mathfrak{m}$. Let us analyse the situation in more detail.

\begin{proposition}
\label{prp:action_translation}
Let $g \in C^k([t_a,t_b],G)$, $k \geq 1$ such that $\dot{g}(t) = T_e \mathcal{R}_{g(t)} \eta(t)$, $\forall t \in [t_a,t_b]$. Denoting $x(t) = \sigma(g(t),x_0) \in M$ then:
\begin{equation*}
D_1 \sigma_{(g(t),x_0)}(\dot{g}(t)) = D_1 \sigma_{(e,x(t))}(\eta(t)),
\end{equation*}
where $D_i$ denotes differentiation with respect to the $i$-th argument.
\end{proposition}

\begin{proof}
To see this, choose $h \in C^k([t_a,t_b],G)$, $k \geq 1$,
\begin{equation*}
h(t) = \left\lbrace
\begin{array}{rl}
g(t), & t_a \leq t < t_c\,,\\
\tilde{g}(t) g(t_c) & t_c \leq t \leq t_b \,,
\end{array}
\right. \quad \text{with } t_c \in [t_a,t_b]\,.
\end{equation*}
In order to have the right continuity, $\tilde{g}(t_c) = e$ and $\dot{\tilde{g}}(t_c) = T_{g(t_c)} \mathcal{R}_{g^{-1}(t_c)} \dot{g}(t_c) = \eta(t_c)$. This means that, on the one hand,
\begin{equation*}
\left.\frac{\mathrm{d}}{\mathrm{d}t} \right\vert_{t = t_c} \sigma(h(t),x_0) = D_1 \sigma_{(h(t_c),x_0)}(\dot{h}(t_c)) = D_1 \sigma_{(g(t_c),x_0)}(\dot{g}(t_c))\,,
\end{equation*}
but on the other,
\begin{align*}
\left.\frac{\mathrm{d}}{\mathrm{d}t} \right\vert_{t = t_c} \sigma(h(t),x) &= \left.\frac{\mathrm{d}}{\mathrm{d}t} \right\vert_{t = t_c} \sigma(\tilde{g}(t) g(t_c),x) = \left.\frac{\mathrm{d}}{\mathrm{d}t} \right\vert_{t = t_c} \sigma(\tilde{g}(t), \sigma(g(t_c),x_0))\\
&= \left.\frac{\mathrm{d}}{\mathrm{d}t} \right\vert_{t = t_c} \sigma(\tilde{g}(t), x(t_c)) = D_1 \sigma_{(e, x(t_c))} (\eta(t_c))\,.
\end{align*}
Since this is true for all $t_c \in [t_a,t_b]$, this concludes our proof.
\end{proof}

Prop.~\ref{prp:action_translation} says that the tangent vector $\eta_M(x(t)) \in T_{x(t)} M$ associated with a curve $x = \sigma(g,x_0) \in C^k(I,M)$ at any given $t \in I$ is directly related with $\eta(t) \in \mathfrak{g}$. Thus, assuming in particular a constant $\eta \in \mathfrak{h}$, since $H$ is the isotropy group of every point of $M$, we have that
\begin{equation*}
\eta_M(x(t)) = D_1 \sigma_{(e, x(t))} (\eta) = \left.\frac{\mathrm{d}}{\mathrm{d}s} \right\vert_{s = 0} \sigma(\exp(s \eta), x(t)) = \left.\frac{\mathrm{d}}{\mathrm{d}s} \right\vert_{s = 0} \sigma(e, x(t)) = 0.
\end{equation*}

This means that only elements in $\mathfrak{m}$ generate actual motions on $M$, and thus it makes sense to restrict to curves on $g(t)$ whose tangent lie on the chosen distribution, $\mathcal{D} = \bigsqcup_{g \in G} \mathcal{D}_g$, provided by the decomposition.\\

Had we chosen a representation of the curve $g \in C^k(I,G)$ in terms of $\dot{g}(t) = T_e \mathcal{L}_{g(t)} \mathrm{H}(t)$ similar arguments would be also valid. For this, an analogous result to Prop.~\ref{prp:action_translation} holds, for which we omit the proof:
\begin{proposition}
\label{prp:action_translation_left}
Let $g \in C^k([t_a,t_b],G)$, $k \geq 1$ such that $\dot{g}(t) = T_e \mathcal{L}_{g(t)} \mathrm{H}(t)$, $\forall t \in [t_a,t_b]$. Then:
\begin{equation*}
D_1 \sigma_{(g(t),x_0)}(\dot{g}(t)) = D_2 \sigma_{(g(t),x_0)}( D_1 \sigma_{(e,x_0)}(\mathrm{H}(t)) )
\end{equation*}
\end{proposition}

\begin{remark}
The map $\widehat{\sigma} : G \times TM \to TM$ is called the lifted action induced by $\sigma$ on the tangent bundle of $M$. Clearly, given $v \in T_x M$, and $g \in G$ we have that $D_2 \sigma(g,x)(v)$ represents the behaviour of $\widehat{\sigma}$ on the fibres.
\end{remark}

Thus, if $\mathrm{H} \in \mathfrak{h}$ no motion is induced on $M$, so it makes sense to restrict to $\mathrm{H} \in \mathfrak{m}$.

\section{Variational problems}
\label{sec:variational_problems}
A (first-order) Lagrangian system on $M$, see \cite{AbrahamMarsden}, is defined by a Lagrangian function $L: TM \to \mathbb{R}$ which is $C^\ell$, with $\ell \geq 2$, and integrable on $I$. One defines the associated action $\mathcal{J}: C^\ell(I,M) \to \mathbb{R}$ as
\begin{equation*}
\mathcal{J}[x] = \int_{I} L(\hat{x}(t))\, \mathrm{d}t
\end{equation*}
where $\hat{x} \in C^{\ell-1}(I,M)$ is the tangent lift of $x$ to $TM$, locally $\hat{x}(t) = (x^i(t),\dot{x}^i(t))$, $i = 1,...,\dim M$. Making use of the homogeneous space structure of $M$, by defining an origin, we can define the pullback of $L$ to $G$, $L^G: TG \to \mathbb{R}$, by $L^G = L \circ T\pi$ \cite{LeeLeokMcClamroch_Global}. This in turn, allows us to define an associated action $\mathcal{J}^G: C^\ell(I,G) \to \mathbb{R}$,
\begin{equation*}
\mathcal{J}^G[g] = \int_{I} L^G(\hat{g}(t)) \, \mathrm{d}t = \int_{I} (L \circ T\pi)(\hat{g}(t)) \, \mathrm{d}t\,.
\end{equation*}

Notice that, since the isotropy group $H$ is not required to be trivial, in general this Lagrangian will be singular as $\dim M \neq \dim G$. Locally, this means that the matrix formed by the coefficients
\begin{equation*}
\frac{\partial^2 L}{\partial \dot{g}^i \partial \dot{g}^j}, \quad i,j = 1,...,\dim G
\end{equation*}
is singular. However, once the problem is constrained so that $\dot{g}(t)$ is a left or right translation of an element $\eta(t) \in \mathfrak{m}$ the problem becomes regular, at least in the continuous case. If $\left\lbrace e_j\right\rbrace_{j = 0}^{\dim H}$, forms a basis of $\mathfrak{h}$, the nonholonomic constraints we need to impose on the algebra are
\begin{equation*}
\phi_j(\eta(t)) := B(\eta(t),e_j) = 0, \quad \forall j = 1, ..., \dim H\,.
\end{equation*}

%Applying Chetaev's principle to $\mathcal{J}^G[g]$ we obtain either of the following:
%\begin{align*}
%\left\langle T^*_e \mathcal{L}_{g} \left[ \frac{\partial L^G}{\partial g^i} - \frac{\mathrm{d}}{\mathrm{d} t}\left( \frac{\partial L^G}{\partial \dot{g}^i}\right)\right] ,  \mathrm{Z} \right\rangle &= 0\,, \quad \mathrm{Z} \in \mathfrak{m}\,;\\
%\left\langle T^*_e \mathcal{R}_{g} \left[ \frac{\partial L^G}{\partial g^i} - \frac{\mathrm{d}}{\mathrm{d} t}\left( \frac{\partial L^G}{\partial \dot{g}^i}\right)\right] ,  \zeta \right\rangle &= 0\,, \quad \zeta \in \mathfrak{m}\,,
%\end{align*}
%depending on the choice of left or right trivialization respectively. This is equivalent to
%\begin{align*}
%\frac{\mathrm{d}}{\mathrm{d} t}\left( \frac{\partial L^G}{\partial \dot{g}^i}\right) - \frac{\partial L^G}{\partial g^i} &= \lambda^j \frac{\partial \Phi_j}{\partial \dot{g}^i}\,,\\
%\Phi_i(g,\dot{g}) &= 0\,,
%\end{align*}
%where $\Phi_i(g,\dot{g}) := \phi_i( T_g \mathcal{L}_{g^{-1}}\dot{g},e_i)$ or $\Phi_i(g,\dot{g}) := \phi_i( T_g \mathcal{R}_{g^{-1}}\dot{g},e_i)$, depending on the chosen trivialization, and the $\lambda^i$ are Lagrange multipliers.
Applying Chetaev's principle with fixed ends to $\mathcal{J}^G[g]$ we obtain
\begin{equation*}
\left\langle \left[ \frac{\partial L^G}{\partial g^i} - \frac{\mathrm{d}}{\mathrm{d} t}\left( \frac{\partial L^G}{\partial \dot{g}^i}\right)\right] ,  \delta g \right\rangle = 0\,, \quad \forall \delta g \in \mathcal{D}\,.
\end{equation*}

This is equivalent to
\begin{align*}
\frac{\mathrm{d}}{\mathrm{d} t}\left( \frac{\partial L^G}{\partial \dot{g}^i}\right) - \frac{\partial L^G}{\partial g^i} &= \lambda^j \frac{\partial \Phi_j}{\partial \dot{g}^i}\,,\quad i = 1,...,\dim G, \, j = 1,..., \dim H,\\
\Phi_j(g,\dot{g}) &= 0\,,
\end{align*}
where $\Phi: TG \to \mathbb{R}^{\dim H}$ is a constraint function such that $\Phi(\mathcal{D}) = 0$, and $\lambda^j$, are Lagrange multipliers. Another variant of these equations can be obtained if we write them in terms of the original Lagrangian:
\begin{align*}
D_1 \sigma_{(g,x_0)}^* \left[ \frac{\mathrm{d}}{\mathrm{d} t}\left( \frac{\partial L}{\partial \dot{x}^i}\right) - \frac{\partial L}{\partial x^i} \right] &= \lambda^j \frac{\partial \Phi_j}{\partial \dot{g}^i}\,,\\
\Phi_j(g,\dot{g}) &= 0\,,
\end{align*}

Moreover, we can define trivialized Lagrangians $\ell^G: G \times \mathfrak{g} \to \mathbb{R}$ by left or right trivialization, i.e. $\ell^G(g,\mathrm{H}) = L^G(g,T_e \mathcal{L}_{g} \mathrm{H})$ or $\ell^G(g,\eta) = L^G(g,T_e \mathcal{R}_{g} \eta)$ respectively. See for instance \cite{BouRabeeMarsden, Holm2}. In that case, the Euler-Lagrange equations can be rewritten as either
\begin{subequations}
\label{eq:Euler_Poincare_nonholon}
\begin{align}
\frac{\mathrm{d}}{\mathrm{d} t}\left( \frac{\partial \ell^G}{\partial \mathrm{H}^i}\right) - \mathrm{ad}_{\mathrm{H}}^*\, \frac{\partial \ell^G}{\partial \mathrm{H}^i} - T^*_e \mathcal{L}_{g} \frac{\partial \ell^G}{\partial g^i} &= \lambda^j \frac{\partial \phi_j}{\partial \mathrm{H}^i}\,,{eq:Euler_Poincare_nonholon}\\
\phi_j(\mathrm{H}) &= 0\,;\\
\frac{\mathrm{d}}{\mathrm{d} t}\left( \frac{\partial \ell^G}{\partial \eta^i}\right) + \mathrm{ad}_{\eta}^*\, \frac{\partial \ell^G}{\partial \eta^i} - T^*_e \mathcal{R}_{g} \frac{\partial \ell^G}{\partial g^i} &= \lambda^j \frac{\partial \phi_j}{\partial \eta^i}\,,\\
\phi_j(\eta) &= 0\,.
\end{align}
\end{subequations}
depending on the choice of left or right trivialization respectively.\\

If variations at the boundary are allowed, these contribute as
\begin{equation*}
\left.\left\langle (D_1 \sigma_{(g,x_0)})^* \frac{\partial L}{\partial \dot{x}}(x,\dot{x}), \delta g\right\rangle\right\vert_{\partial I} = \left.\left\langle \frac{\partial L^G}{\partial \dot{g}}(g,\dot{g}), \delta g\right\rangle\right\vert_{\partial I} = \left.\left\langle \frac{\partial \ell^G}{\partial \mathrm{H}}(x,\mathrm{H}), \mathrm{Z}\right\rangle\right\vert_{\partial I} = \left.\left\langle \frac{\partial \ell^G}{\partial \eta}(x,\eta), \zeta \right\rangle\right\vert_{\partial I}
\end{equation*}
where $\partial I$ represents the boundary of the interval $I$, $\mathrm{Z} := \left(\mathcal{L}_g^{-1}\right)_* \delta g$ and $\zeta := \left(\mathcal{R}_g^{-1}\right)_{*} \delta g$.

\section{Variational integrators}
\label{sec:variational_integrators}
Let us discretize a curve $g \in C^\ell(I,G)$, with $\ell \geq 2$ and $h > 0$, by a discrete curve $g_d : I_d \to G$ where $I_d$ is a constant step-size discretization of $I$ with step size $h$. This gives us a set of points $\left\lbrace g_k \right\rbrace_{k = 0}^N$, steps, where $N$ is the number of intervals of length $h$ in which $I$ is discretized.\\

Assume a (tangent) retraction \cite{AbsMahSep} on $G$, i.e. a smooth map $\tau: \mathfrak{g} \to G$, such that $\tau(0) = e$, $T_0 \tau \equiv \mathrm{Id}_{\mathfrak{g}}$, that is a diffeomorphism on its image. The exponential map is a retraction, but there may be others such as the Cayley map. Then, a discrete curve on a Lie group can be parametrized using left or right trivialization by elements on $\mathfrak{g}$ via such retractions, e.g. $g_{k+1} = g_k \tau(\xi_k)$ or $g_{k+1} = \tau(\xi_k) g_k$, respectively.\\

Moreover, each element $\xi_k$ can be parametrized by an internal subdivision in \emph{so-called} stages $\left\lbrace \Xi_k^i\right\rbrace_{i = 1}^s$. We call these elements stage variables and $s$ is the number of stages of the subdivision. With such a subdivision, we can obtain approximations of $g$ over each interval, namely $G_k^i = g_k \tau(\Xi_k^i)$ (left) or $G_k^i = \tau(\Xi_k^i) g_k$ (right). The coupling of this discretization with the application of a RKMK algorithm allows us to obtain approximations to solutions of ODEs on $G$ \cite{MuntheKaas,CeMaOw}.\\

It is common to reserve capital letters for stage variables, therefore, in this section we will stop using capitalization to distinguish between left and right trivialized magnitudes and we will simply separate both cases.\\

With the action of the group on $M$, we can generate discrete curves $x_d : I_d \to M$, $x_{k+1} = \sigma(g_k,x_{k})$ and stage subdivisions $X_k^i = \sigma(G_k^i, x_k)$.\\

We have seen in the previous section that a first order variational problem requires working with the tangent lift of curves. Using the tangent retraction in the continuous case we may write $g$ in terms of a curve $\xi \in C^\ell(I,\mathfrak{g})$, $g(t) = \tau(\xi(t))$. To compute the tangent lift of such a curve, we need the derivatives of the retraction, $\dot{g}(t) = T\tau_{\xi(t)} (\dot{\xi}(t))$. Using left or right trivialization we can write this as
\begin{equation*}
\dot{g}(t) = T_e \mathcal{L}_{\tau(\xi(t))} (\mathrm{d}^L \tau_{\xi(t)} \dot{\xi}(t)) = T_e \mathcal{R}_{\tau(\xi(t))} (\mathrm{d}^R \tau_{\xi(t)} \dot{\xi}(t))\,,
\end{equation*}
where $\mathrm{d}^L \tau, \mathrm{d}^R \tau : \mathfrak{g} \times \mathfrak{g} \to \mathfrak{g}$ are the left and right trivialized tangents to the exponential respectively \cite{IserlesLie,CeMaOw,BouRabeeMarsden}, which can be interpreted as translation operations on the algebra induced by another element of the algebra. Comparing this with our computations in previous sections, we can identify $\mathrm{H} := \mathrm{d}^L \tau_{\xi} \dot{\xi}$ and $\eta := \mathrm{d}^R \tau_{\xi} \dot{\xi}$.\\

A suitable discretization of $\dot{g}$ at the stage level is
\begin{equation*}
\dot{G}_k^i = T_e \mathcal{L}_{g_k \tau(\Xi_k^i)} (\mathrm{d}^L \tau_{\Xi_k^i} \dot{\Xi}_k^i)\,, \quad \text{or} \quad \dot{G}_k^i = T_e \mathcal{R}_{\tau(\Xi_k^i) g_k} (\mathrm{d}^R \tau_{\Xi_k^i} \dot{\Xi}_k^i)\,.
\end{equation*}
%Notice that, contrary to what happens with $\xi$ in the continuous case, the $\Xi_k^i$ are not the same in the left and right trivialized cases because $G_k^i \neq \tau(\Xi_k^i)$ (they are actually related by $\mathrm{Ad}_{g_k}$).\\

One can then obtain a discretization of $\dot{x}$ as
\begin{equation*}
\dot{X}_k^i = D_1 \sigma_{(G_k^i,x_0)} \left( ( \mathcal{L}_{G_k^i} )_* (\mathrm{d}^L \tau_{\Xi_k^i} \dot{\Xi}_k^i) \right)\,, \quad \text{or} \quad \dot{X}_k^i = D_1 \sigma_{(G_k^i,x_0)} \left( ( \mathcal{R}_{G_k^i} )_* (\mathrm{d}^R \tau_{\Xi_k^i} \dot{\Xi}_k^i) \right)\,,
\end{equation*}
where $G_k^i$ and the pushforward notation have been used to shorten these expressions.\\

Assume now a given RK method, defined by the coefficients $(a_{i j}, b_j, c_i)$, $i,j = 1,...,s$ that form its Butcher tableaux. See for instance \cite{HaLuWa}. The derivation of a variationally partitioned RKMK integrator is based on the discretization of the action using a quadrature rule
\begin{equation*}
\mathcal{J}_d[x_d] = \sum_{k = 0}^{N-1} h \sum_{i = 1}^s b_i L(X_k^i, \dot{X}_k^i)\,,
\end{equation*}
or in terms of $L^G$,
\begin{equation*}
\mathcal{J}_d^G[g_d] = \sum_{k = 0}^{N-1} h \sum_{i = 1}^s b_i (L \circ T\pi)(G_k^i, \dot{G}_k^i) = \sum_{k = 0}^{N-1} h \sum_{i = 1}^s b_i L^G (G_k^i, \dot{G}_k^i)\,.
\end{equation*}
This action must be supplemented with the RKMK constraints
\begin{subequations}
\label{eq:RKMK_constraints}
\begin{align}
\tau^{-1}(g_k^{-1} G_{k}^{i}) &= \Xi_{k}^i = \sum_{j = 1}^s a_{i j} \dot{\Xi}_{k}^{j}\,,\\
\tau^{-1}(g_k^{-1} g_{k+1}) &= \xi_{k} = \sum_{j = 1}^s b_{j} \dot{\Xi}_{k}^{j}\,,
\end{align}
\end{subequations}
as well as the nonholonomic constraints to fully define the problem.

\subsection{Discrete holonomically constrained Hamilton-Pontryagin action}
\label{ssec:discrete_hamilton_pontryagin}
In this section we will focus on a left-trivialized problem. Similar expressions can be obtained for the right-trivialized case.\\

We are going to derive the purely variational equations of a discrete Lagrangian problem on a Lie group. In order to simplify the transition between the variational and the nonholonomic case in the following section, we are going to consider a holonomically constrained system. This means that we will assume a Lagrangian $L^G: TG \to \mathbb{R}$ and a constraint submanifold $N \subset G$. This submanifold can be handled by the inclusion of a constraint function $\Phi^{\mathrm{hol}}: G \to V$, $V \cong \mathbb{R}^{\mathrm{codim}_G N}$, such that $\Phi^{\mathrm{hol}}(N) = 0$.\\

The discrete variational equations of the holonomically constrained variational problem are those of a variationally partitioned RKMK method \cite{BouRabeeMarsden,BogfjellmoMarthinsen,NonholonomicMartinSato18}. These can be obtained by choosing a discrete space of curves
\begin{align*}
&C_d^s (I_d, g_a, g_b)\\
&= \left\lbrace \left(g,\kappa,\left\lbrace \Xi^i, \dot{\Xi}^i, \mathrm{K}^i, \Lambda^i \right\rbrace_{i = 1}^s\right) \, : \, I_d \rightarrow G \times \mathfrak{g}^* \times \left(\mathbb{T}\mathfrak{g} \times V \right)^s \,|\, g(t_0) = g_a, g(t_N) = g_b\right\rbrace\,,
\end{align*}
with $\mathbb{T}\mathfrak{g} \cong \mathfrak{g} \times \mathfrak{g} \times \mathfrak{g}^*$, and a discrete Hamilton-Pontryagin functional on the group:
\begin{align*}
(\mathcal{J_{HP}}^G)_d &= \sum_{k = 0}^{N-1} \sum_{i = 1}^s h b_i \left[ \vphantom{\sum_{j = 1}^s} L^G\left(g_k \tau(\Xi_k^i), \left( \mathcal{L}_{g_k \tau(\Xi_k^i)} \right)_* \mathrm{d}^L \tau_{\Xi_k^i} \dot{\Xi}_k^i\right) + \left\langle\!\!\left\langle \Lambda^i_k, \Phi^{\mathrm{hol}}(g_k \tau(\Xi_k^i))\right\rangle\!\!\right\rangle \right.\\
&+ \left. \left\langle \mathrm{K}_k^i, \frac{1}{h}\Xi_k^i - \sum_{j = 1}^s a_{i j} \dot{\Xi}_k^j\right\rangle + \left\langle \kappa_{k+1}, \frac{1}{h} \tau^{-1}( g_k^{-1} g_{k+1}) - \sum_{j = 1}^s b_{i j} \dot{\Xi}_k^j\right\rangle\right]\,,
\end{align*}
where $\left\langle\!\!\!\left\langle \cdot, \cdot \right\rangle\!\!\!\right\rangle: V \times V \to \mathbb{R}$ is the Euclidean inner product in $V$.

%This can be rewritten in terms of the Lagrangian on $M$ and the action of the Lie group:
%\begin{align*}
%(\mathcal{J_{HP}})_d &= \sum_{k = 0}^{N-1} \sum_{i = 1}^s h b_i \left[ \vphantom{\sum_{j = 1}^s} L\left(\sigma(g_k \tau(\Xi_k^i),x_0), D_1\sigma_{(g_k \tau(\Xi_k^i),x_0)} \left( \left( \mathcal{L}_{g_k \tau(\Xi_k^i)} \right)_* \mathrm{d}^L \tau_{\Xi_k^i} \dot{\Xi}_k^i \right) \right)\right.\\
%&+ \left. \left\langle \mathrm{K}_k^i, \frac{1}{h}\Xi_k^i - \sum_{j = 1}^s a_{i j} \dot{\Xi}_k^j\right\rangle + \left\langle \kappa_{k+1}, \frac{1}{h} \tau^{-1}( g_k^{-1} g_{k+1}) - \sum_{j = 1}^s b_{i j} \dot{\Xi}_k^j\right\rangle\right]\,.
%\end{align*}

Before proceeding to vary the action, let us introduce some maps that will appear in these computations:
\begin{itemize}
\item The second left-trivialized tangent to the exponential, $\mathrm{dd}^L \tau : \mathfrak{g} \times \mathfrak{g} \times \mathfrak{g} \to \mathfrak{g}$, defined by $\zeta \partial_{\xi} \left( \mathrm{d}^L \tau_\xi \eta \right) = \mathrm{dd}^L \tau_{\xi}(\eta,\zeta)$.
\item The derivative of the left translation operator defined by $\frac{\mathrm{d}}{\mathrm{d}t}\left.\left( \mathcal{L}_{g(t)} \eta \right)\right\vert_{\dot{g} = \mathcal{L}_{g}\zeta} = D \mathcal{L}_g (\eta,\zeta)$, $\eta,\zeta \in \mathfrak{g}$.
\item A second derivative of the Lie group action $\frac{\mathrm{d}}{\mathrm{d}t}\left.\left( D_1 \sigma_{(g(t),x_0)} v \right)\right\vert_{\dot{g} = w} = D_{1 1} \sigma_{(g,x_0)} (v,w)$, $v, w \in T_g G$.
\end{itemize}

With these we can define:
\begin{itemize}
\item $T \mathrm{d}^L \tau : T T \mathfrak{g} \to \mathfrak{g}^4$, $T \mathrm{d}^L \tau (\xi,\dot{\xi}, V_\xi, V_{\dot{\xi}}) = (\xi, \mathrm{d}^L \tau_{\xi} \dot{\xi}, \mathrm{d}^L \tau_{\xi} V_{\xi}, \mathrm{d}^L \tau_{\xi} V_{\dot{\xi}} + \mathrm{dd}^L \tau_{\xi}(\dot{\xi},V_{\xi}) ) = (\xi, \eta, V_0, V_{\eta})$. We interpret $\eta \in T_{0}\mathfrak{g}$ and $(V_0, V_{\eta}) \in T_{(0, \eta)} T \mathfrak{g}$.
\item $T \mathcal{L}: G \times \mathfrak{g}^3 \mapsto TTG$, $T \mathcal{L}(g,\zeta,\xi,\dot{\zeta}) = (g, \mathcal{L}_g \zeta, \mathcal{L}_g \xi, \mathcal{L}_g \dot{\zeta} + D \mathcal{L}_{g}(\zeta,\xi)) = (g,g',\dot{g},\dot{g}')$.
\item $\widetilde{\sigma}: TTG \times M \mapsto TTM$, $\widetilde{\sigma}(g,\dot{g},g',\dot{g}',x_0) = (\sigma(g,x_0), D_1 \sigma_{(g,x_0)}(\dot{g}), D_1 \sigma_{(g,x_0)}(g'), D_1 \sigma_{(g,x_0)}(\dot{g}') +$\\
$D_{1 1} \sigma_{(g,x_0)}(\dot{g},g')) = (x,\dot{x},x',\dot{x}')$.
\end{itemize}

Obviously, the variation with respect to the Lagrange multipliers $\kappa$ and $\mathrm{K}^i$ trivially gives us the aforementioned RKMK constraints (Eq.~\eqref{eq:RKMK_constraints}) by construction. Similarly, variation with respect to the Lagrange multipliers $\Lambda^i$ gives the constraints. Defining the trivialized group variations as $\zeta := \left( \mathcal{L}_{g^{-1}} \right)_* \delta g$ and using the short-hand $\left( D_j L^G \right)_k^i := D_j L^G\left(g_k \tau(\Xi_k^i), \left( \mathcal{L}_{g_k \tau(\Xi_k^i)} \right)_* \mathrm{d}^L \tau_{\Xi_k^i} \dot{\Xi}_k^i\right)$ and analogously for $D(\Phi^{\mathrm{hol}})_k^i$, the non-trivial variation terms are:
\begin{align*}
\zeta_k &: \sum_{k = 0}^{N-1} h \sum_{i = 1}^s b_i \left[ \left\langle \left(\mathrm{Ad}_{-\tau(\Xi_k^i)}\right)^* \mathcal{L}_{g_k \tau(\Xi_k^i)}^* \left[ \left( D_1 L^G \right)_k^i + \left\langle\!\!\left\langle \Lambda^i_k, D(\Phi^{\mathrm{hol}})^i_k \right\rangle\!\!\right\rangle \right] \right.\right.\\
&+ \left(\mathrm{Ad}_{\tau(-\Xi_k^i)}\right)^* \left[ \left(D\mathcal{L}_{g_k \tau(\Xi_k^i)}\right)_* \left( \mathrm{d}^L \tau_{\Xi_k^i} \dot{\Xi}_k^i \right)\right]^* \left( D_2 L^G \right)_k^i\\
&- \left.\left. \frac{1}{h} \left(\mathrm{Ad}_{\tau\left(-\xi_{k}\right)}\right)^*\left(\mathrm{d}^{L}\tau_{\xi_{k}}^{-1}\right)^* \kappa_{k+1}, \zeta_{k} \vphantom{\left(\mathrm{Ad}_{\tau(-\Xi_k^i)}\right)^*} \right\rangle \right]\,.\\
\zeta_{k+1} &: \sum_{k = 0}^{N-1} h \sum_{i = 1}^s b_i \left[ \left\langle \frac{1}{h} \left(\mathrm{d}^{L}\tau_{\xi_{k}}^{-1}\right)^*\kappa_{k+1}, \zeta_{k+1} \right\rangle \right]\,.\\
\delta \Xi &: \sum_{k = 1}^{N-1} h \sum_{i = 1}^s b_i \left[ \vphantom{\frac{1}{h}} \left\langle \left( \mathrm{d}^L \tau_{\Xi_k^i} \right)^* \mathcal{L}_{g_k \tau(\Xi_k^i)}^* \left[ \left( D_1 L^G \right)_k^i + \left\langle\!\!\left\langle \Lambda^i_k, D(\Phi^{\mathrm{hol}})^i_k \right\rangle\!\!\right\rangle \right] \right.\right.\\
&+ \left( \mathrm{d}^L \tau_{\Xi_k^i} \right)^* \left[ \left(D\mathcal{L}_{g_k \tau(\Xi_k^i)}\right)_* \left( \mathrm{d}^L \tau_{\Xi_k^i} \dot{\Xi}_k^i \right) \right]^* \left( D_2 L^G \right)_k^i\\
&+ \left.\left. \left( \mathrm{dd}^L \tau_{\Xi_k^i} \dot{\Xi}_k^i \right)^* \mathcal{L}_{g_k \tau(\Xi_k^i)}^* \left( D_2 L^G \right)_k^i, \delta \Xi_k^i \right\rangle + \frac{1}{h} \left\langle  \mathrm{K}_k^i, \delta \Xi_k^i \right\rangle \right]\,.\\
\delta \dot{\Xi} &: \sum_{k = 1}^{N-1} h \sum_{i = 1}^s b_i \left[ \vphantom{\sum_{j = 1}^s} \left\langle \left( \mathrm{d}^L \tau_{\Xi_k^i} \right)^* \mathcal{L}_{g_k \tau(\Xi_k^i)}^* \left( D_2 L^G \right)_k^i, \delta \dot{\Xi}_k^i \right\rangle - \left\langle \mathrm{K}_k^i, \sum_{j = 1}^s a_{i j} \delta \dot{\Xi}_k^j \right\rangle - \left\langle \kappa_{k+1}, \sum_{j = 1}^s b_{j} \delta \dot{\Xi}_k^j \right\rangle \right]\,.
\end{align*}

The first two variations can be rearranged into one single term running through $k = 1,...,N-1$ plus a boundary term at $k = 0$ from the first and another from the latter at $k = N$. Identifying these boundary terms with those found in the continuous case we get that $\mu_k := \left(\mathrm{d}^{L}\tau^{-1}_{\xi_{k-1}}\right)^* \kappa_k = (\mathcal{L}_{g_k})^* D_2 L^G(g_k,\dot{g}_{k})$.\\

Using the following definitions
%\begin{align*}
%{\mathrm{N}_0}_k^i &= \mathcal{L}_{g_k \tau(\Xi_k^i)}^* \left[ \left( D_1 L^G \right)_k^i + \left\langle\!\!\left\langle \Lambda^i_k, D\Phi(G^i_k) \right\rangle\!\!\right\rangle \right] + \left[ \left(D\mathcal{L}_{g_k \tau(\Xi_k^i)}\right)_* \left( \mathrm{d}^L \tau_{\Xi_k^i} \dot{\Xi}_k^i \right)\right]^* \left( D_2 L^G \right)_k^i,\\
%{\mathrm{M}_0}_k^i &= \mathcal{L}_{g_k \tau(\Xi_k^i)}^* \left( D_2 L^G \right)_k^i,\\
%\mathrm{N}_k^i &= \left( \mathrm{d}^{L}\tau_{\Xi_k^i} \right)^* {\mathrm{N}_0}_k^i + \left( \mathrm{dd}^{L}\tau_{\Xi_k^i} \dot{\Xi}_k^i \right)^* {\mathrm{M}_0}_k^i,\\
%\mathrm{M}_k^i &= \left( \mathrm{d}^{L}\tau_{\Xi_k^i} \right)^* {\mathrm{M}_0}_k^i,\\
%\widehat{X}_k &= \left(\mathrm{d}^{L}\tau^{-1}_{\xi_{k}}\right)^* X_k\,,
%\end{align*}
\begin{align*}
\mathrm{d}(L^G)^{i}_k &= \left(\left( D_1 L^G \right)_k^i , \left( D_2 L^G \right)_k^i\right) \in T^*_{(g_k \tau(\Xi_k^i), ( \mathcal{L}_{g_k \tau(\Xi_k^i)} )_* \mathrm{d}^L \tau_{\Xi_k^i} \dot{\Xi}_k^i)} TG\\
(F_{\Phi^{\mathrm{hol}}})^{i}_k &= \left( \left\langle\!\!\left\langle \Lambda^i_k, D(\Phi^{\mathrm{hol}})^i_k \right\rangle\!\!\right\rangle, 0 \right) \in T^*_{(g_k \tau(\Xi_k^i), ( \mathcal{L}_{g_k \tau(\Xi_k^i)} )_* \mathrm{d}^L \tau_{\Xi_k^i} \dot{\Xi}_k^i)} TG\\
( {\mathrm{N}_0}_k^i, {\Pi_0}_k^i ) &= T^* \mathcal{L} \left[ \mathrm{d}(L^G)^{i}_k + (F_{\Phi^{\mathrm{hol}}})^{i}_k \right] \in T^*_{(0,\mathrm{d}^L \tau_{\Xi_k^i} \dot{\Xi}_k^i)} T \mathfrak{g}\\
( {\mathrm{N}}_k^i, {\Pi}_k^i ) &= T^* \mathrm{d}^L \tau ( {\mathrm{N}_0}_k^i, {\Pi_0}_k^i ) \in T^*_{(\Xi_k^i, \dot{\Xi}_k^i)} T \mathfrak{g}\\
\widehat{X}_k &= \left(\mathrm{d}^{L}\tau^{-1}_{\xi_{k}}\right)^* X_k\,,
\end{align*}
we finally get the discrete equations that conform our variationally partitioned RKMK method \footnote{Python libraries containing all the necessary operators for the Lie groups $SO(3)$ and $SE(3)$ by this author can be found at \url{https://github.com/RodriTaku/LieGroupsPython}.}:
\begin{subequations}
\label{eq:VPRKMK}
\begin{align}
\Xi_k^i &= \tau^{-1}\left(g_k^{-1} G_k^i\right) = h \sum_{j = 1}^s a_{i j} \dot{\Xi}_k^j,\\
\xi_{k} &= \tau^{-1}\left(g_k^{-1} g_{k+1}\right) = h \sum_{j = 1}^s b_{j} \dot{\Xi}_k^j,\\
\widehat{\Pi}_k^i &= \mathrm{Ad}_{\tau(\xi_{k})}^* \left[\mu_k + h \sum_{j = 1}^s b_j \left( \mathrm{Ad}_{\tau(-\Xi_{k}^j)}^* {\mathrm{N}_0}_{k}^j - \frac{a_{j i}}{b_i} \mathrm{Ad}_{\tau(-\xi_{k})}^* \left.\widehat{\mathrm{N}}\right._{k}^j \right) \right],\\
\mu_{k+1} &= \mathrm{Ad}_{\tau(\xi_{k})}^* \left[ \mu_{k} + h \sum_{j = 1}^s b_j \mathrm{Ad}_{\tau(-\Xi^j_k)}^* {\mathrm{N}_0}_k^j\right],\label{eq:VPRKMK_final_momentum}\\
0 &= \Phi^{\mathrm{hol}}(G^i_k)\,.\label{eq:VPRKMK_hol_constr}
\end{align}
\end{subequations}

Using matrices, we can write in the above expressions
\begin{align*}
T^*\mathcal{L}
&= \left[\begin{array}{cc}
\mathcal{L}_{g_k \tau(\Xi_k^i)}^* & \left[ \left(D\mathcal{L}_{g_k \tau(\Xi_k^i)}\right)_* \left( \mathrm{d}^L \tau_{\Xi_k^i} \dot{\Xi}_k^i \right)\right]^* \\
0 & \mathcal{L}_{g_k \tau(\Xi_k^i)}^*
\end{array}
\right]\,,\\
T^*\mathrm{d}^L \tau
&= \left[\begin{array}{cc}
\left( \mathrm{d}^{L}\tau_{\Xi_k^i} \right)^* & \left( \mathrm{dd}^{L}\tau_{\Xi_k^i} \dot{\Xi}_k^i \right)^* \\
0 & \left( \mathrm{d}^{L}\tau_{\Xi_k^i} \right)^*
\end{array}
\right]\,.
\end{align*}

Following \cite{MarsdenWest}, Sec. 3.5.6, to warrant both solvability and that not only the stage variables but also the step variables satisfy the constraints we require that the RK method be stiffly accurate, i.e. $a_{s j} = b_j$, and $a_{1 j} = 0$, which are methods of Lobatto-type \cite{SatoMartindeAlmagro2021}. This means that the first and last stage variables coincide with the step variables. However, this renders Eq.~\eqref{eq:VPRKMK_hol_constr} redundant for $i = 1$, since the initial stage variable satisfies the constraint.\\

Since we need another equation to have a locally unique solution, we may add a further constraint. The natural choice, since consistent initial-step values must satisfy this, is the tangency constraint at the final step value \cite{Jay96,MarsdenWest}
\begin{equation*}
\left\langle D \Phi^{\mathrm{hol}}(g_{k+1}), \dot{g}_{k+1} \right\rangle = 0
\end{equation*}
where the velocity $\dot{g}_{k+1}$ is implicitly defined by $\mu_{k+1} = (\mathcal{L}_{g_{k+1}})^* D_2 L^G(g_{k+1},\dot{g}_{k+1})$.\\

\begin{remark}
If a trivialized Lagrangian $\ell^G$ is used, we can write
\begin{align*}
\mathrm{d}(\ell^G)^{i}_k &= \left(\left( D_1 \ell^G \right)_k^i , \left( D_2 \ell^G \right)_k^i\right) \in T^*_{(g_k \tau(\Xi_k^i), \mathrm{d}^L \tau_{\Xi_k^i} \dot{\Xi}_k^i)} (G \times \mathfrak{g})\\
(f_{\Phi})^{i}_k &= \left( \left\langle\!\!\left\langle \Lambda^i_k, D(\Phi^{\mathrm{hol}})^i_k \right\rangle\!\!\right\rangle, 0 \right) \in T^*_{(g_k \tau(\Xi_k^i), \mathrm{d}^L \tau_{\Xi_k^i} \dot{\Xi}_k^i)} (G \times \mathfrak{g})\,,\\
( {\mathrm{N}_0}_k^i, {\Pi_0}_k^i ) &= \widetilde{\mathcal{L}}^* \left[ \mathrm{d}(\ell^G)^{i}_k + (f_{\Phi^{\mathrm{hol}}})^{i}_k \right],
\end{align*}
with
\begin{equation*}
\widetilde{\mathcal{L}}^*
= \left[\begin{array}{cc}
\mathcal{L}_{g_k \tau(\Xi_k^i)}^* & 0 \\
0 & I
\end{array}
\right]\,.
\end{equation*}

If the original Lagrangian $L$ is used, then
\begin{equation*}
\mathrm{d}(L^G)^{i}_k = \tilde{\sigma}^*_{(x_0)} \mathrm{d}(L)^{i}_k \,,
\end{equation*}
with
\begin{equation*}
\tilde{\sigma}^*_{x_0}
= \left[\begin{array}{cc}
D_1 \sigma_{(g_k \tau(\Xi_k^i),x_0)}^*  & \left[ D_{1 1} \sigma_{(g_k \tau(\Xi_k^i),x_0)} \left( \left( \mathcal{L}_{g_k \tau(\Xi_k^i)} \right)_* \mathrm{d}^L \tau_{\Xi_k^i} \dot{\Xi}_k^i \right)\right]^* \\
0 & D_1 \sigma_{(g_k \tau(\Xi_k^i),x_0)}^*
\end{array}
\right]\,.
\end{equation*}
\end{remark}

\subsection{Nonholonomic extension}
As stated at the beginning of Sec.~\ref{ssec:discrete_hamilton_pontryagin}, the former equations are only valid for purely variational problems. Moreover, as mentioned in Sec.~\ref{sec:variational_problems}, the Lagrangians $L^G$ and $\ell^G$ will generally be singular. We will focus on the trivialized case, but the non-trivialized case is entirely analogous.\\

We need to include the nonholonomic constraints imposed by the distribution in the discrete setting. However, contrary to what happens in the continuous case, the resulting discrete nonholonomic equations are not well-defined unless a regular Lagrangian in $G$ is provided. This is unfortunate, but not too difficult to overcome. Any regular Lagrangian $\ell^{\mathrm{reg}}: G \times \mathfrak{g} \to \mathbb{R}$ satisfying $\left.\ell^{\mathrm{reg}}\right\vert_{\mathcal{D}} = \ell^{G}$ will give the correct dynamics.\\

Once a suitable $\ell^{\mathrm{reg}}$ is chosen in place of $\mathcal{l}^G$, we proceed to modify the algorithm to include the nonholonomic constraints. First, the nonholonomic constraint force to be used in place of the holonomic one is
\begin{equation*}
(f_{\Phi})^{i}_k = \left( \left\langle\!\!\left\langle \Lambda^i_k, D_2\Phi\left( g_k \tau(\Xi_k^i), \left( \mathcal{L}_{g_k \tau(\Xi_k^i)} \right)_* \mathrm{d}^L \tau_{\Xi_k^i} \dot{\Xi}_k^i \right) \right\rangle\!\!\right\rangle, 0 \right) \in T^*_{(g_k \tau(\Xi_k^i), \mathrm{d}^L \tau_{\Xi_k^i} \dot{\Xi}_k^i)} (G \times \mathfrak{g})\,,
\end{equation*}
However, since the reduced constraint $\phi: \mathfrak{g} \to \mathbb{R}$ is at our disposal, we can use it
\begin{equation*}
(f_{\phi})^{i}_k = \left( \left\langle\!\!\left\langle \Lambda^i_k, D\phi\left( \mathrm{d}^L \tau_{\Xi_k^i} \dot{\Xi}_k^i \right) \right\rangle\!\!\right\rangle, 0 \right)\,,
\end{equation*}
in which case no transport by $\widetilde{\mathcal{L}}^*$ needs to be applied, i.e. 
\begin{equation*}
( {\mathrm{N}_0}_k^i, {\Pi_0}_k^i ) = \widetilde{\mathcal{L}}^* \mathrm{d}(\ell^G)^{i}_k + (f_{\phi})^{i}_k\,.
\end{equation*}

Second, we need to impose the nonholonomic constraint itself, Eq.~\eqref{eq:VPRKMK_hol_constr}. For this we need to substitute the holonomic constraint by the nonholonomic one. However, this poses the important question of where should one impose this constraint.\\

Nonholonomic constraints do not only depend on the configuration variables but also on the fibres, be it velocities or momenta. The discrete velocities $\mathrm{d}^L \tau_{\Xi_k^i} \dot{\Xi}^i_k$ are not as good approximations of the velocities $\eta(t_k + c_i h)$ as $G^i_k = g_k \tau(\Xi^i_k)$ are of the configuration variables $g(t_k + c_i h)$. Similarly $\Pi^i_k$ are not as good approximations of the momenta $\mu(t_k + c_i h) = D_2 \ell^G(g(t_k + c_i h), \eta(t_k + c_i h))$.\\

We want to impose the constraints closer to the actual constraint manifold. For this, as proposed in \cite{SatoMartindeAlmagro2021,NonholonomicMartinSato18}, we introduce new velocity approximations $\mathrm{H}^i_k$ and momenta $\mathrm{M}^i_k = D_2 \ell^G(G^i_k, \mathrm{H}^i_k)$, with which we can impose the constraint:
\begin{equation}
\label{eq:VPRKMK_nonhol_constr}
\phi(\mathrm{H}^i_k) = 0\,, \quad i = 1,...,s
\end{equation}

Naturally, Eqs.~(\ref{eq:VPRKMK}a-d) and \eqref{eq:VPRKMK_nonhol_constr} need to be supplied with new equations to fully define $\mathrm{H}^i_k$. For this, we have
\begin{equation}
\label{eq:VPRKMK_momenta}
{\mathrm{M}_0}^i_k = \mathrm{Ad}_{\tau(\Xi^i_k)}^* \left[ \mu_{k} + h \sum_{j = 1}^s a_{i j} \mathrm{Ad}_{\tau(-\Xi^j_k)}^* {\mathrm{N}_0}_k^j\right]\,, \quad i = 1,...,s.
\end{equation}
Notice that, due to the properties of the RK coefficients, for $i = 1$ we have that ${\mathrm{M}_0}^1_k = \mu_k$ and for $i = s$ we recuperate Eq.~\eqref{eq:VPRKMK_final_momentum}. Thus this set of equations can be regarded as an extension of Eq.~\eqref{eq:VPRKMK_final_momentum} to inner stages. Actually, if we were to consider the free case, i.e. the ${\mathrm{N}_0}_k^i$ do not contain constraint forces, this equation could be applied to a symplectic integrator as a post-processing layer, providing us with approximations of the inner momenta and velocities of the same order as the $G^i_k$.\\

Finally, a crucial issue must be addressed. Similar to what happens in the holonomic case, Eq.~\eqref{eq:VPRKMK_nonhol_constr} for $i = 1$ is redundant since we assume that each initial-step value satisfies the constraint. However, in the nonholonomic case we do not have a natural choice of additional constraint to impose. In this regard, it can be argued that this algorithm is incomplete. This issue is intimately related with the results of \cite{PerlmutterMcLachlan}, Sec.7.\\

A tangency constraint in this case would either impose values on the final Lagrange multipliers $\Lambda^s_k$, thus breaking the symmetry of the method, or add new multipliers to be related to the rest through additional equations that would need to be specified. The first possibility is similar to another strategy where $\Lambda^1_k$, instead of $\Lambda^s_k$, for each $k$ are assumed to be derived from the expression for the multipliers obtained in the continuous setting. This has been explored and its performance varies from problem to problem, some times providing excellent results while some others displaying disappointing long-term energy behaviour. Another possibility would be instead to impose some other kind of constraint, perhaps related to the dynamics of the multipliers.\\

Instead of adding further equations, we commonly assume that the last multipliers of the current step are inserted as the firsts of the next, i.e. $\Lambda^1_{k+1} = \Lambda^s_k$. We call this $\Lambda$-\emph{concatenation}. In this case we also assume that $\Lambda^1_0$ coincides with its continuous counterpart which can be readily computed from initial data by differentiation of the constraint. This seems to be a common choice for similar algorithms, e.g. \cite{Jay93}, and appears to give good results in terms of energy preservation in many situations. However, this is still an unacceptable state of affairs and finding the adequate equations is an active research topic of this author. This choice is also unfortunately linked to some stability issues of this algorithm, as we will see in the examples in the next section. Hopefully in the future these issues can be overcome by finding the right constraints to impose.\\

\section{Application to homogeneous spaces and numerical tests}\label{eq:numerical_tests}
As an example we consider the case of the sphere, $M = S^2$, and the special orthogonal group, $G = SO(3)$, which acts transitively on it. Since any rotation of the sphere leaves the pair of points lying on the intersection with the axis of rotation fixed, the isotropy group of any point is therefore an element of $SO(2)$. Thus $S^2 \cong SO(3)/SO(2)$.\\

$S^2$ can be embedded in $\mathbb{R}^3$, which will be particularly useful to write the expressions for the Lagrangian of these systems. Similarly, the action of the group, $\sigma: SO(3) \times S^2 \to S^2$, can be easily handled by using a matrix representation of the group. Thus, if $g \in SO(3) \subset M_{3}(\mathbb{R})$, we can simply write
\begin{equation*}
\sigma(g,x_0) = g x_0\,,
\end{equation*}
where on the right-hand side we have the multiplication of a $3 \times 3$ square matrix by a column matrix of dimension $3$. We may choose the origin $x_0$ at our own discretion. We will take $x_0 = (0,0,1)$, the north pole of the sphere.\\

The algebra of $SO(3)$, $\mathfrak{g} = \mathfrak{so}(3)$ can be thought of as the set of $3 \times 3$ skew-symmetric matrices. This latter space is isomorphic to $\mathbb{R}^3$ and therefore its elements can also be thought of as column vectors. In this regard, we have that the infinitesimal action at the identity becomes
\begin{equation*}
D_1 \sigma_{(e,x_0)}(\eta) = \eta \times x_0\,,
\end{equation*}
where $\times$ denotes the standard cross product in Euclidean space. We can readily see from this that $\mathfrak{h} = \left\lbrace \eta \in \mathfrak{so}(3) \, \vert \, \eta \propto x_0 \right\rbrace \cong \mathfrak{so}(2)$. Therefore, with our choice of $x_0$ we have that $\mathfrak{m} = \left\lbrace \eta \in \mathfrak{so}(3) \, \vert \, \eta \cdot x_0 = 0\right\rbrace$, where $\cdot$ denotes the standard Euclidean inner product. We should unpack here the fact that $B: \mathfrak{so}(3) \times \mathfrak{so}(3) \to \mathbb{R}$ has been identified with the Euclidean inner product and, given how $\mathfrak{h}$ is defined, $x_0$ is regarded here as one of its elements.\\

In what follows, we apply our integrator to two Lagrangian systems: the spherical pendulum and the Kepler problem on the sphere.

\subsection{Spherical mathematical pendulum}
The Lagrangian of this system is $L: T S^2 \subset \mathbb{R}^6 \to \mathbb{R}$, which we may write as
\begin{equation*}
L(x,\dot{x}) = \frac{m}{2} \left\Vert \dot{x} \right\Vert^2 + \gamma \cdot x,
\end{equation*}
where $\left\Vert \right\Vert$ denotes the norm associated to the inner product, $m$ is the mass of the bob of the pendulum and $\gamma \in \mathbb{R}^3$ is the gravitational field intensity.\\

The Lagrangian on $G$, $L^G: TSO(3) \to \mathbb{R}$, is
\begin{equation*}
L^G(g,\dot{g}) = (L \circ \pi)(g, \dot{g}) = L(\sigma(g,x_0),D_1 \sigma_{(g,x_0)}(\dot{g})) = \frac{m}{2} \left\Vert \dot{g} x_0 \right\Vert^2 + \gamma \cdot (g x_0).
\end{equation*}
Using Prop.~\ref{prp:action_translation_left} and taking into account the linearity of the action with respect to its second argument, i.e.
\begin{equation*}
D_2 \sigma_{(g,x_0)} (x_1) = g x_1\,, \quad \forall x_1 \in S^2\,,
\end{equation*}
we get that if $\dot{g} = (\mathcal{L}_{g})_* \eta$, then $\dot{g} x_0 = g(\eta \times x_0)$, and due to the fact that the elements of $SO(3)$ are the isometries of $\mathbb{R}^3$, the left-trivialized Lagrangian $\ell^G: SO(3) \times \mathfrak{so}(3) \to \mathbb{R}$ becomes
\begin{equation*}
\ell^G(g,\eta) = \frac{m}{2} \left\Vert \eta \times x_0 \right\Vert^2 + \gamma \cdot (g x_0)\,.
\end{equation*}

It is not difficult to check that this Lagrangian is singular since the components of $\eta$ parallel to $x_0$ are lost. A very simple regularization is
\begin{equation*}
\ell^{\mathrm{reg}}(g,\eta) = \frac{m}{2} \left\Vert \eta \times x_0 \right\Vert^2 + \frac{M}{2} (\eta \cdot x_0)^2 + \gamma \cdot (g x_0)\,, \quad M \neq 0.
\end{equation*}

Finally, the constraint function can be simply written as $\phi(\eta) = \eta \cdot x_0 = \eta_3$. Applying Eq.~(\ref{eq:Euler_Poincare_nonholon}a-b) since we are considering the left-trivialized case and choosing $\gamma = (0,0,-\alpha)$, the equations of motion are
\begin{align*}
m \dot{\eta}_1 &= (m - M) \eta_3 \eta_2 + \alpha \sin \theta_1 \cos \theta_2\,,\\
m \dot{\eta}_2 &= -(m - M) \eta_3 \eta_1 + \alpha \sin \theta_2\,,\\
M \dot{\eta}_3 &= \lambda\,,\\
        \eta_3 &=  0\,,
\end{align*}
where we have used Tait-Bryan angles $(\theta_1, \theta_2, \theta_3)$ as coordinates in $SO(3)$, with matrices representing a given rotation $R_z(\theta_3) R_y(\theta_2) R_x(\theta_1)$, with
\begin{equation*}
R_x(\theta) = \left(
\begin{array}{ccc}
1 &           0 &            0\\
0 & \cos \theta & -\sin \theta\\
0 & \sin \theta &  \cos \theta
\end{array}
\right)\,,
\quad
R_y(\theta) = \left(
\begin{array}{ccc}
 \cos \theta & 0 & \sin \theta\\
           0 & 1 &           0\\
-\sin \theta & 0 & \cos \theta
\end{array}
\right)\,,
\quad
R_z(\theta) = \left(
\begin{array}{ccc}
\cos \theta & -\sin \theta & 0\\
\sin \theta &  \cos \theta & 0\\
          0 &            0 & 1
\end{array}
\right)\,,
\end{equation*}
and the fact that $\mathrm{ad}_{\xi}^* \mu = \mu \times \xi$. Differentiating the constraint and substituting in the third equation we get that $\lambda = 0$.\\

Before proceeding to apply a nonholonomic method, it should be mentioned that given the simplicity of the system, the conservation of the parallel component of the angular momentum to $\gamma$ would suffice to warrant that we do not leave the constraint manifold. Indeed, application of a symplectic integrator, which automatically preserves these symmetries, would provide excellent results. In spite of this, we will apply a nonholonomic method and we may use this to highlight some of the differences.\\

In Figs.~\ref{fig:pendulum}, \ref{fig:pendulum_energy_evolution} and \ref{fig:pendulum_lambda_evolution} we offer several plots of numerical solutions of the problem for integrators of order 2, 4 and 6 corresponding to the 2, 3 and 4-stage Lobatto methods using the Cayley map as retraction $\tau$. The order of the methods was proven in \cite{SatoMartindeAlmagro2021} in vector spaces, but the result holds in the Lie group setting since this is a geometrically consistent modification. The initial conditions have been chosen as $g_0 = (0, \pi/3, 0)$, $\eta_0 = (1/3,0,0)$ and $\lambda_0 = \Lambda^1_0 = 0$.\\

\begin{figure}[]
	\centering
	\begin{minipage}{0.55\textwidth}
		\centering
		\includegraphics[width=\textwidth]{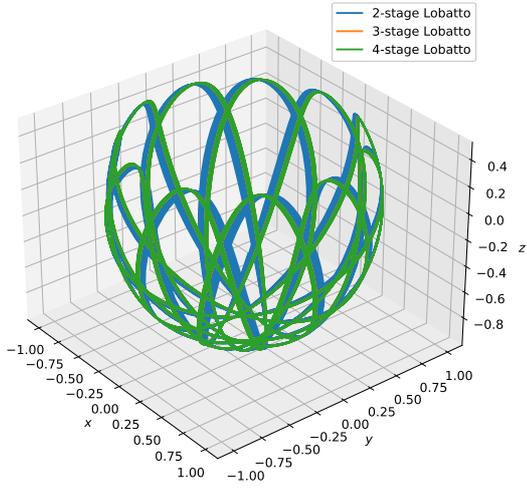}
	\end{minipage}\hfill
	\begin{minipage}{0.4\textwidth}
		\caption{Trajectories of the spherical pendulum with $h = 0.1$ computed with $\tau = \mathrm{cay}$.}\label{fig:pendulum}
	\end{minipage}
\end{figure}

\begin{figure}[]
	\centering
	\begin{subfigure}[b]{0.32\textwidth}
		\centering
		\includegraphics[width=\textwidth]{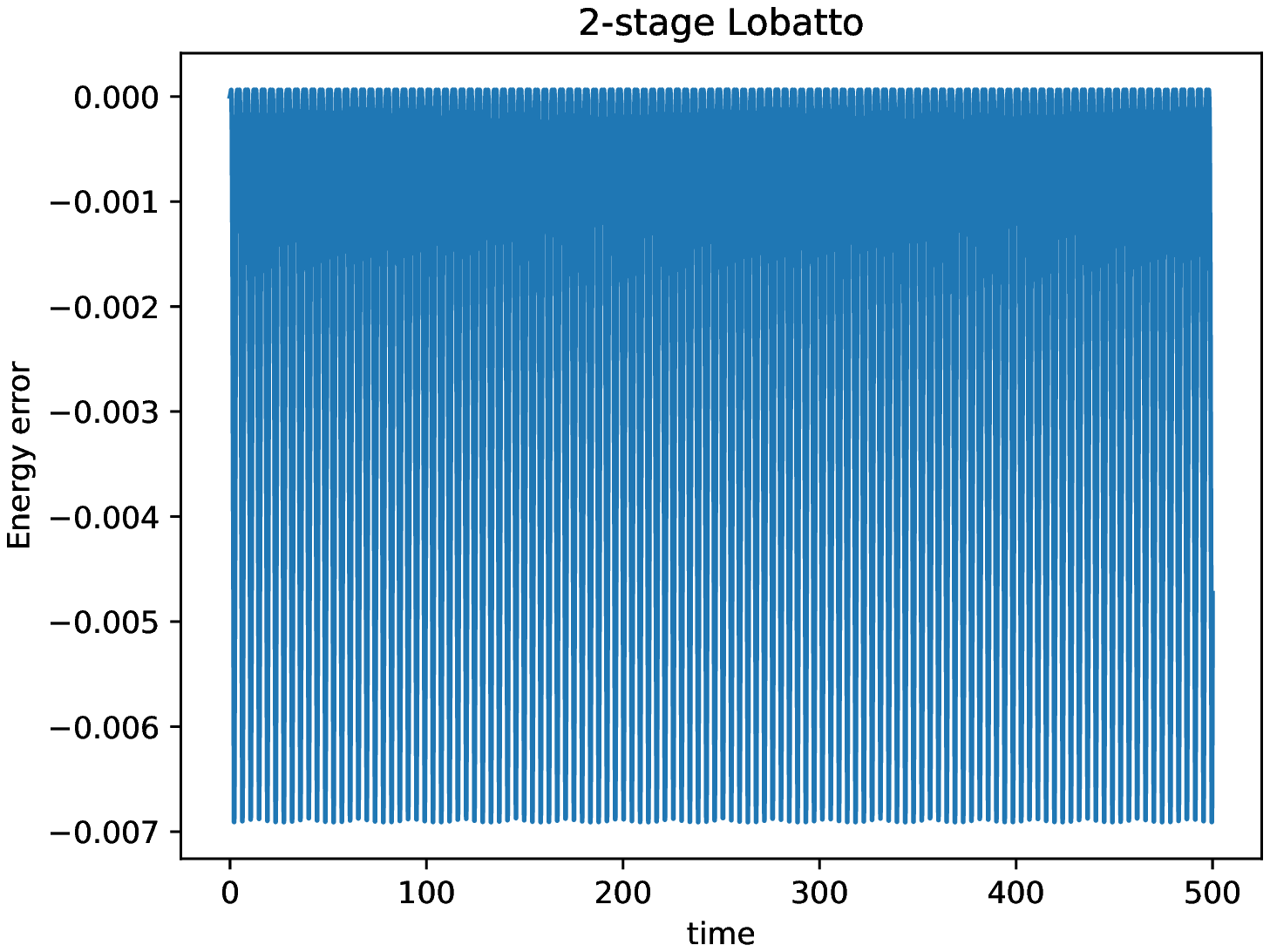}
		\caption{}\label{fig:pendulum_energy_evolution_lobatto2}
	\end{subfigure}
	\begin{subfigure}[b]{0.32\textwidth}
		\centering
		\includegraphics[width=\textwidth]{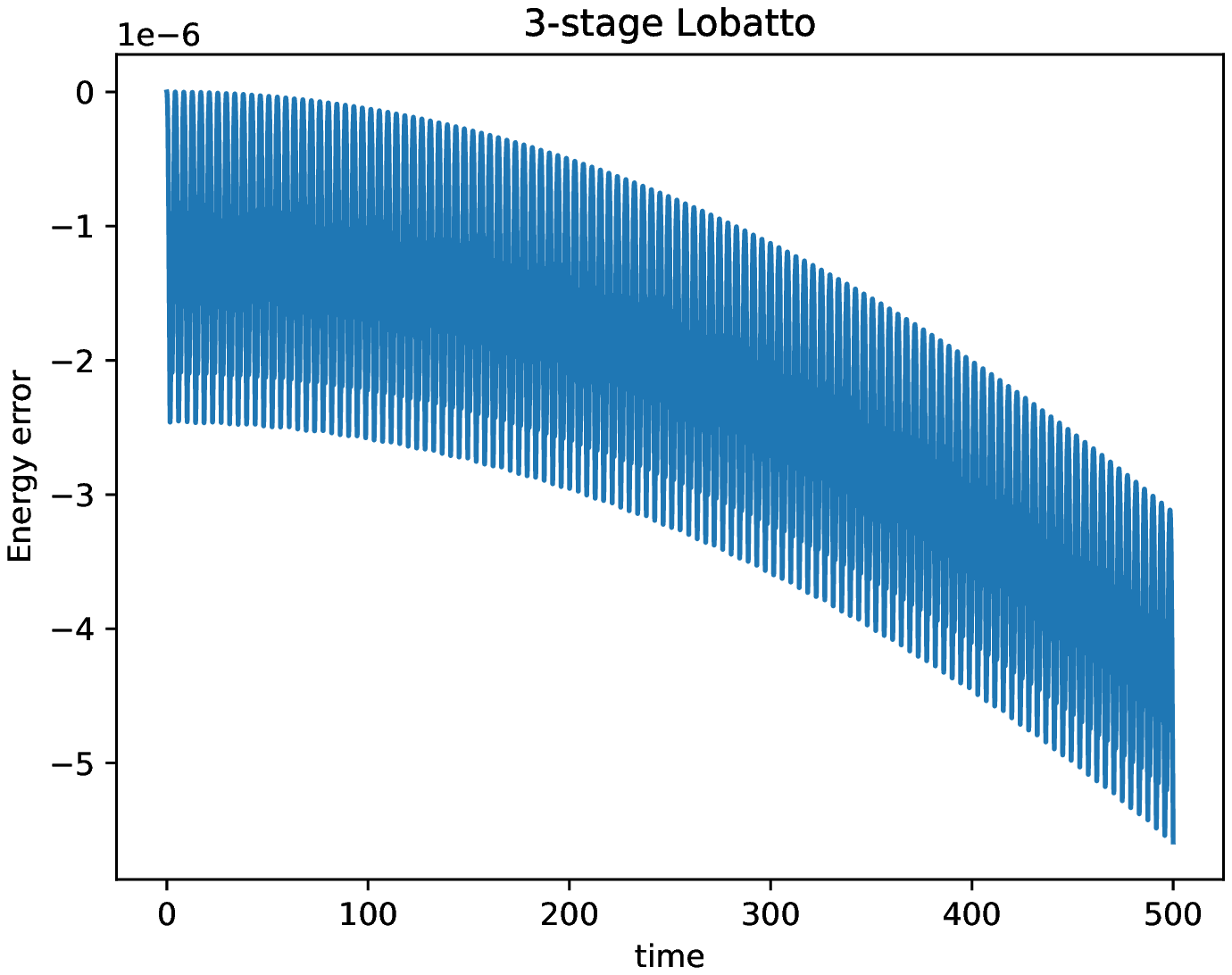}
		\caption{}\label{fig:pendulum_energy_evolution_lobatto3}
	\end{subfigure}
	\begin{subfigure}[b]{0.32\textwidth}
		\centering
		\includegraphics[width=\textwidth]{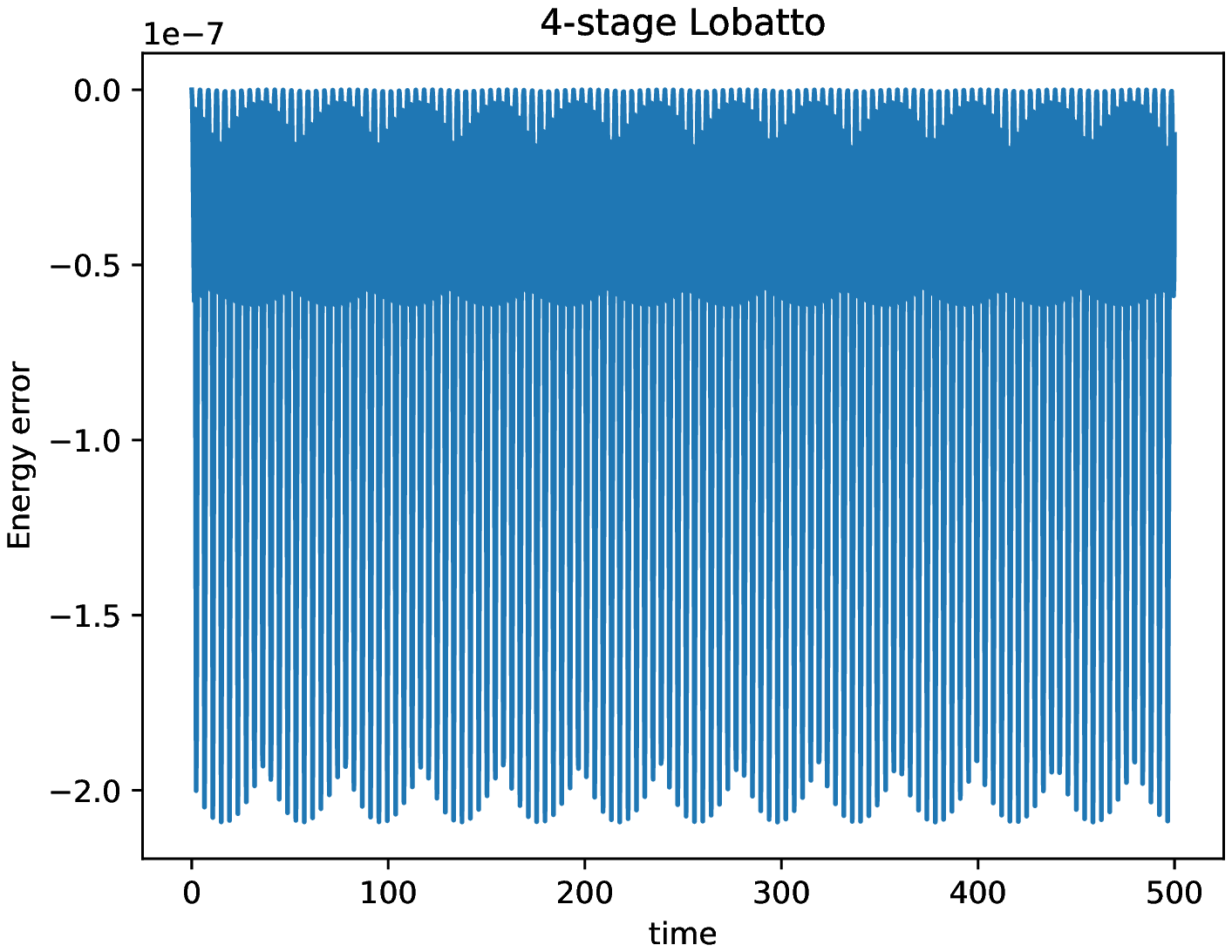}
		\caption{}\label{fig:pendulum_energy_evolution_lobatto4}
	\end{subfigure}
	  \caption{Evolution of the energy error, $E(t)-E(0)$, of the pendulum.}\label{fig:pendulum_energy_evolution}
\end{figure}

For the 3-stage Lobatto method with $\Lambda$-concatenation the system displays some of the instability mentioned before. This issue remains to be rigorously analysed, but a very plausible explanation is as follows. The fact that we have chosen to chain Lagrange multipliers from one step to the next instead of adding further equations to fix them inside each step means on the one hand that we have effectively lost some degrees of freedom, and on the other that the dynamics of the rest of the system is left to cope with whatever the earlier steps have done. All of this means that the system is particularly vulnerable to drift induced by these multipliers.\\

\begin{figure}[]
	\centering
	\begin{subfigure}[b]{0.32\textwidth}
		\centering
		\includegraphics[width=\textwidth]{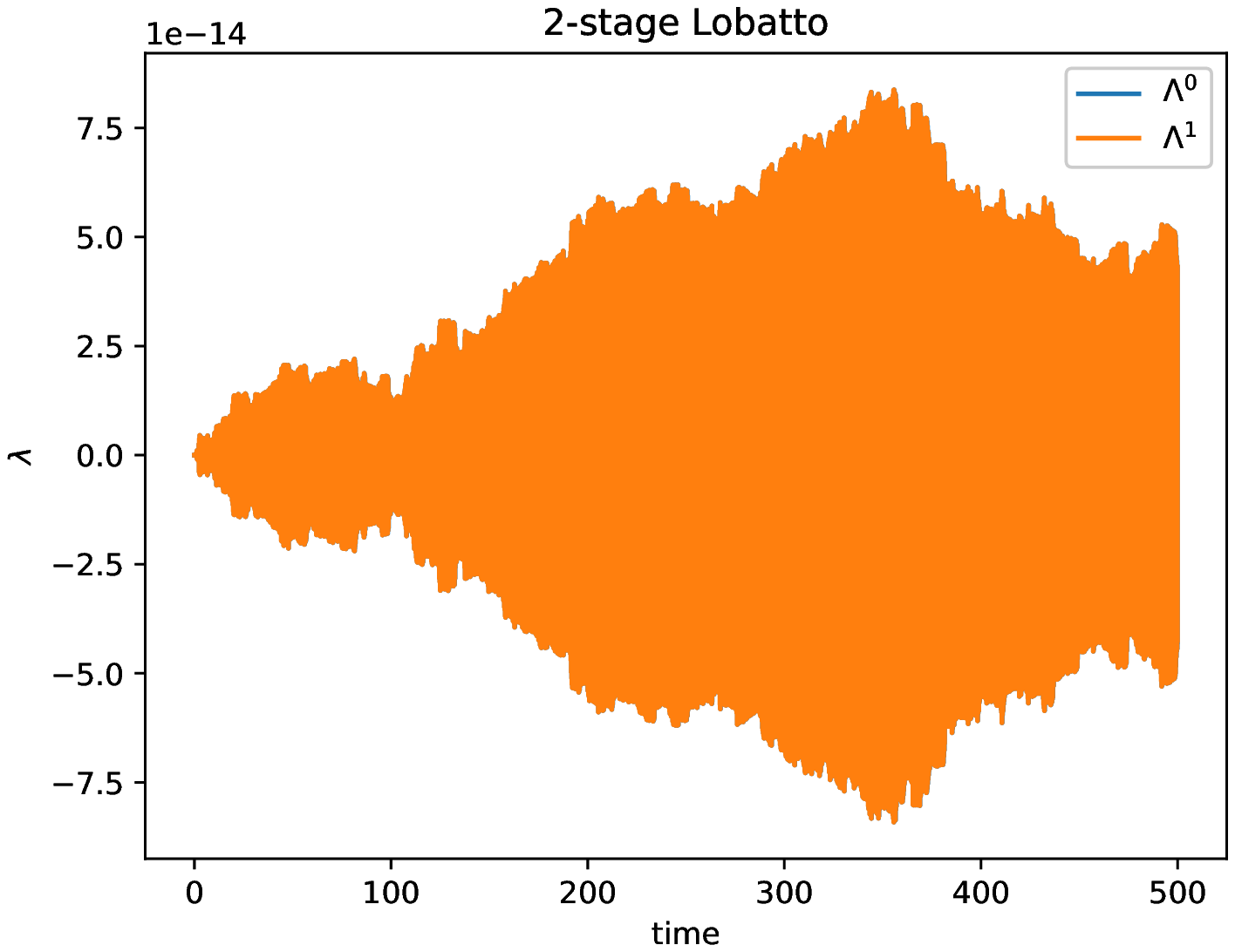}
		\caption{}\label{fig:pendulum_lambda_evolution_lobatto2}
	\end{subfigure}
	\begin{subfigure}[b]{0.32\textwidth}
		\centering
		\includegraphics[width=\textwidth]{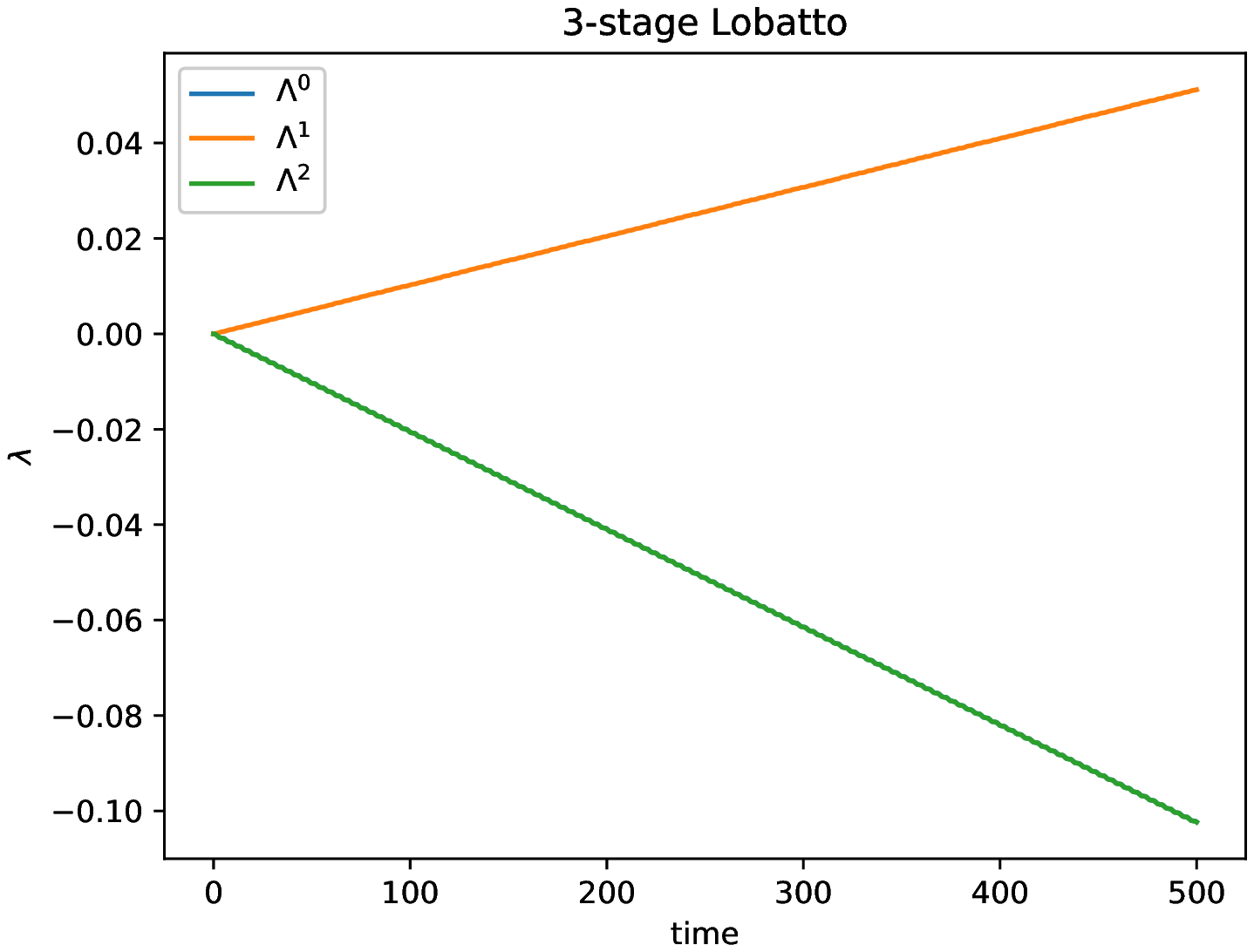}
		\caption{}\label{fig:pendulum_lambda_evolution_lobatto3}
	\end{subfigure}
	\begin{subfigure}[b]{0.32\textwidth}
		\centering
		\includegraphics[width=\textwidth]{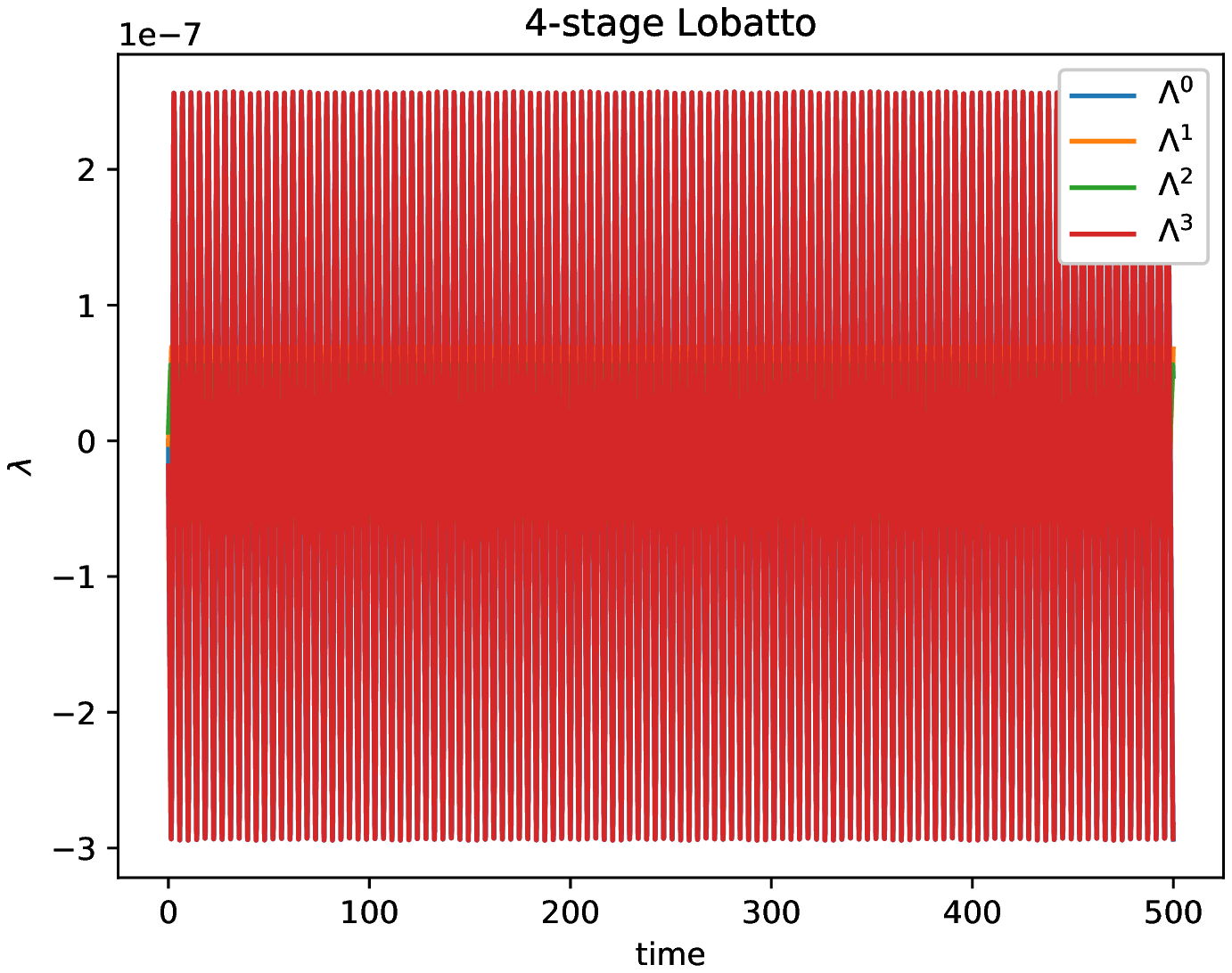}
		\caption{}\label{fig:pendulum_lambda_evolution_lobatto4}
	\end{subfigure}
	  \caption{Evolution of the Lagrange multipliers in the pendulum.}\label{fig:pendulum_lambda_evolution}
\end{figure}

In Figs.~\ref{fig:pendulum_lambda_evolution_lobatto2}, \ref{fig:pendulum_lambda_evolution_lobatto3} and \ref{fig:pendulum_lambda_evolution_lobatto4} we can see that in this case the 2 and 4-stage methods can maintain the dynamics of the multipliers bound but the 3-stage method cannot, and the multipliers drift. The rest of the system maintains the variational-like behaviour for as much as it can, but eventually the increasing amount of perturbation induced by the drifting multipliers will push the system far enough so that the system needs to reconfigure.\\

Instead of $\Lambda$-concatenation we can set $\Lambda^1_k = 0$ in accordance with the continuous case, or impose some other arbitrary equation such as $\sum_{j = 1}^s b_j \Lambda^j_k = 0$. As it turns out, both of these choices work rather well in this case, the latter of which can be seen in Figs.\ref{fig:pendulum_energy_evolution_lobatto3_mod} and \ref{fig:pendulum_lambda_evolution_lobatto3_mod}.\\

\begin{figure}[]
	\centering
	\begin{subfigure}[b]{0.45\textwidth}
		\centering
		\includegraphics[width=\textwidth]{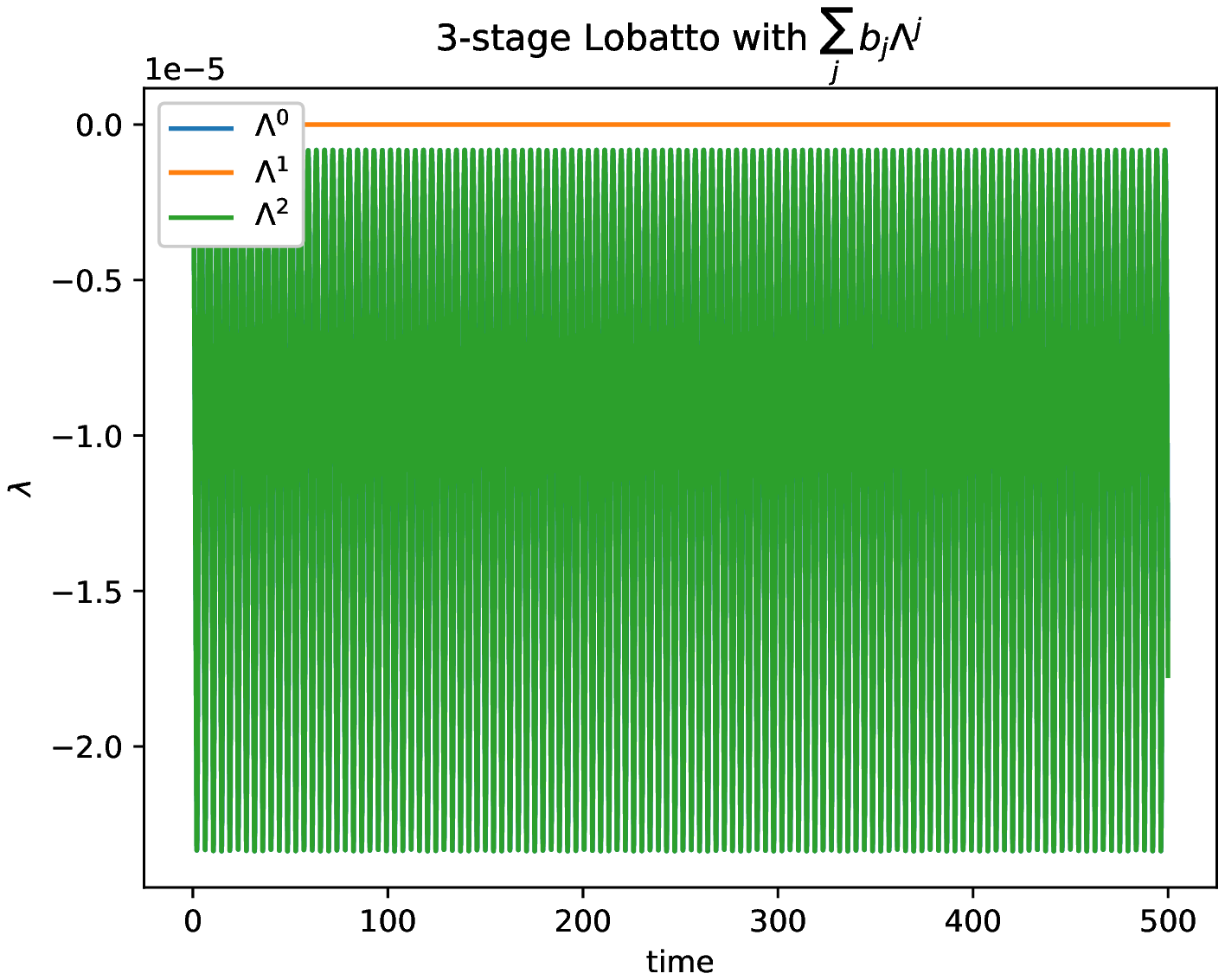}
		\caption{}\label{fig:pendulum_lambda_evolution_lobatto3_mod}
	\end{subfigure}
	\begin{subfigure}[b]{0.45\textwidth}
		\centering
		\includegraphics[width=\textwidth]{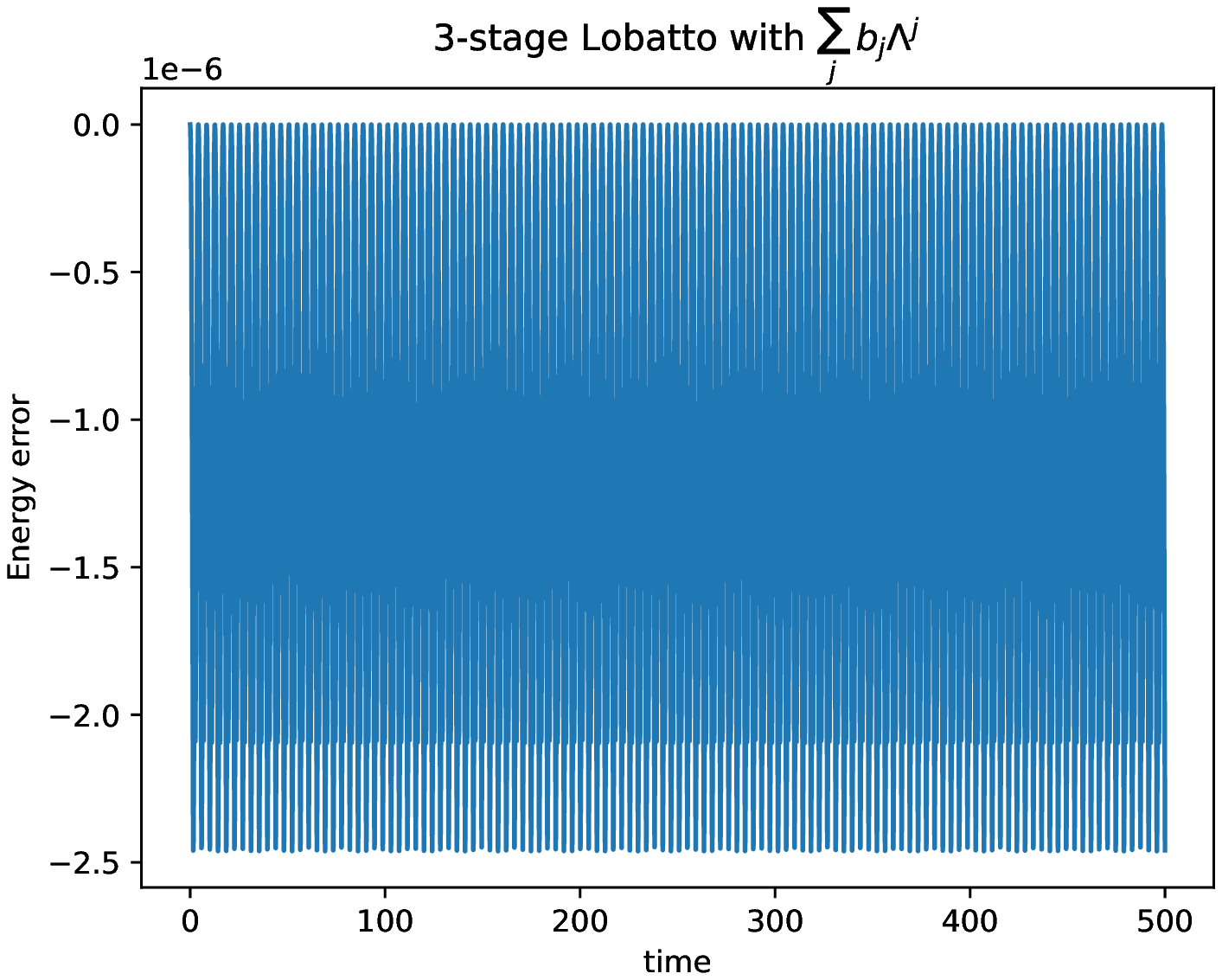}
		\caption{}\label{fig:pendulum_energy_evolution_lobatto3_mod}
	\end{subfigure}
	  \caption{Evolution of the energy and Lagrange multipliers of the pendulum integrated with the 3-stage Lobatto method with additional equation.}\label{fig:pendulum_evolution_mod}
\end{figure}

In Figs.~\ref{fig:order_pendulum_exp} and \ref{fig:order_pendulum_exp_lambda} we can see that the numerical order estimations correspond to the theoretical ones. The order of the multiplier for the 2-stage method is actually better than expected in this case.\\

\begin{figure}[ht]
	\centering
	\begin{subfigure}[b]{0.65\textwidth}
		\centering
		\includegraphics[width=\textwidth]{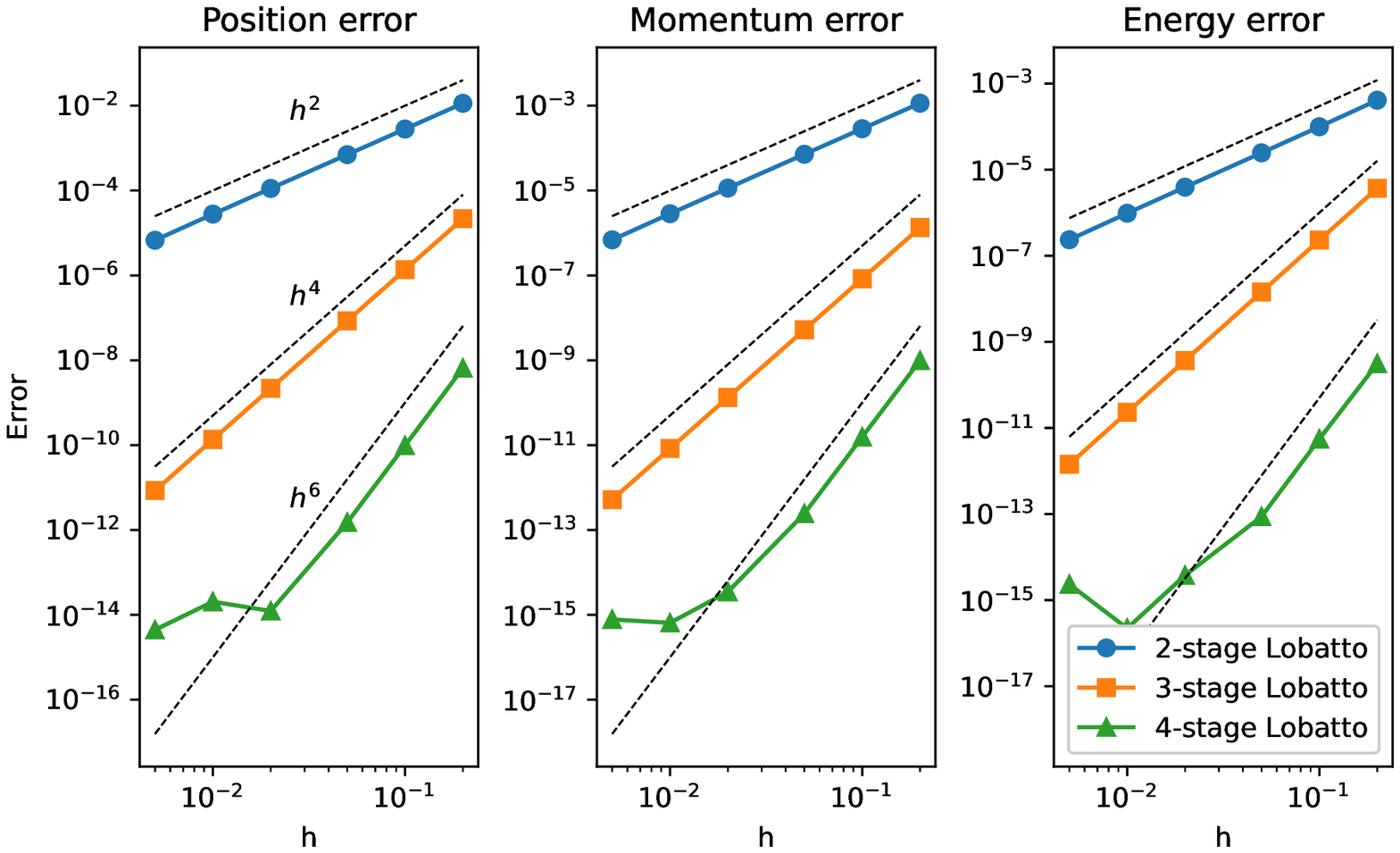}
	  \caption{}\label{fig:order_pendulum_exp}
	\end{subfigure}
	\begin{subfigure}[b]{0.25\textwidth}
		\centering
		\includegraphics[width=4.5cm]{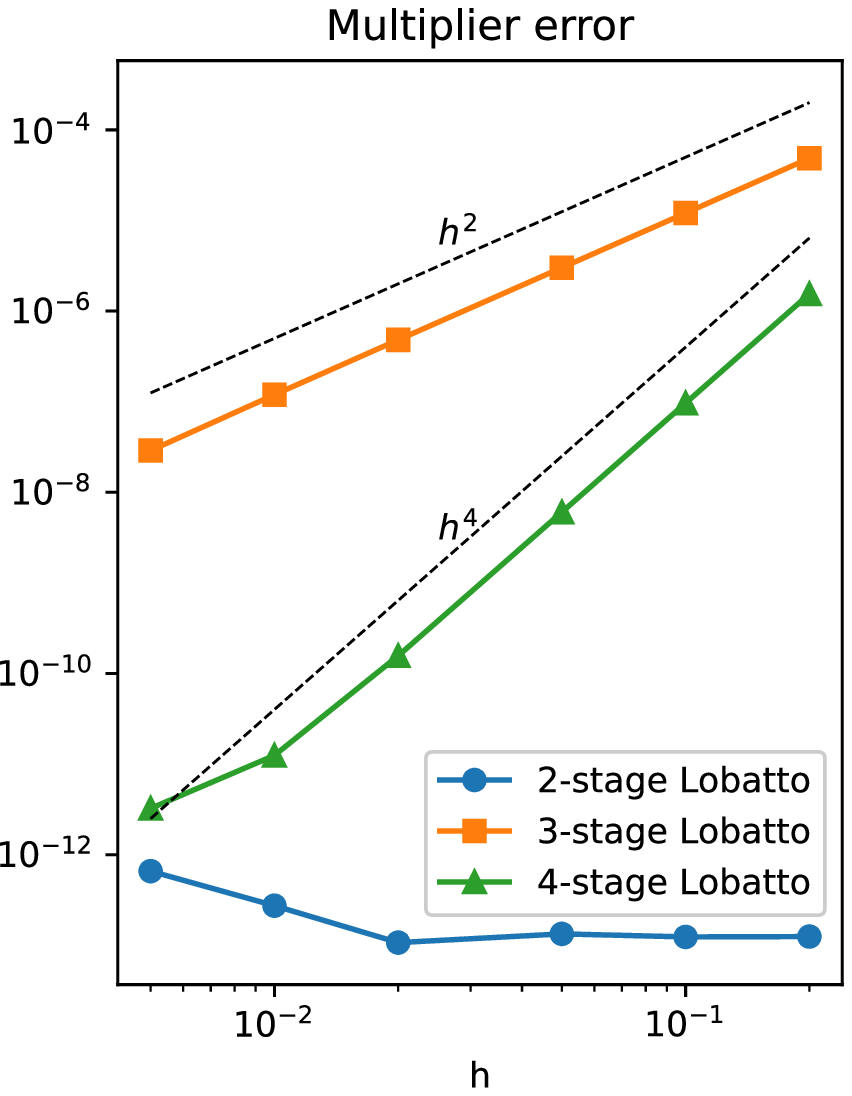}
	  \caption{}\label{fig:order_pendulum_exp_lambda}
	\end{subfigure}
	\caption{Error plots computed for the pendulum system using $\tau = \exp$. The nonlinear systems were solved using \texttt{scipy.optimize.root} with \texttt{tol=1e-14}, which explains the flattening of the curves.}
\end{figure}

As mentioned before, the same system can be integrated using a symplectic integrator without having to impose any constraint. When possible to use, symplectic integrators are superior to nonholonomic methods because we can automatically warrant the good long-term energy behaviour at any order and they do not restrict us to Lobatto-type methods. Moreover, they are computationally less expensive since no extra momentum equations need to be computed, the constraint need not be explicitly imposed and no Lagrange multipliers are used. However, it should be noted that using a symplectic integrator it is not possible to warrant that inner stages satisfy the constraint and in general they do not.

\subsection{Spherical Kepler problem}
The standard Kepler problem studies the dynamics of a particle of mass $m$ and position $x$ in Euclidean space induced by the gravitational potential generated by a massive fixed particle at position $X$. The Kepler potential is inversely proportional to the distance between the particles. The spherical version of the problem considers that the particles live in a spherical geometry and the potential is given by a fundamental solution of the Laplace-Beltrami equation on the sphere \cite{HaLuWa_StormerVerlet}.\\

The Lagrangian of this system is thus
\begin{equation*}
L(x,\dot{x}) = \frac{m}{2} \left\Vert \dot{x} \right\Vert^2 + \rho \frac{X \cdot x}{\sqrt{1 - (X \cdot x)^2}},
\end{equation*}
where $\rho$ parametrizes the strength of the gravitational potential. The trivialized and regularized Lagrangian of this system is
\begin{equation*}
\ell^{\mathrm{reg}}(g,\eta) = \frac{m}{2} \left\Vert \eta \times x_0 \right\Vert^2 + \frac{M}{2} (\eta \cdot x_0)^2 + \rho \frac{X \cdot (g x_0)}{\sqrt{1 - [X \cdot (g x_0)]^2}}\,, \quad M \neq 0.
\end{equation*}

\begin{figure}[]
	\centering
	\begin{minipage}{0.55\textwidth}
	\centering
		\includegraphics[width=\textwidth]{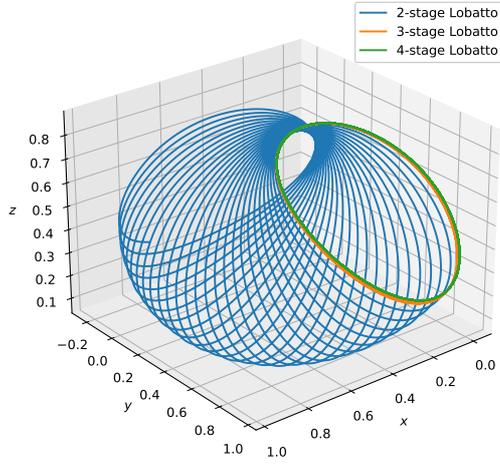}
	\end{minipage} \hfill
	\begin{minipage}{0.4\textwidth}
		\caption{Orbits of the spherical Kepler problem computed with $h = 0.01$. This system displays the characteristic precession of its planar counterpart when applying symplectic methods and for fixed $h$, the lower the order of the method the more noticeable the effect is.}\label{fig:kepler_orbits}
	\end{minipage}
\end{figure}
\begin{figure}[]
	\centering
	\begin{subfigure}[b]{0.5\textwidth}
	\centering
		\includegraphics[width=7cm]{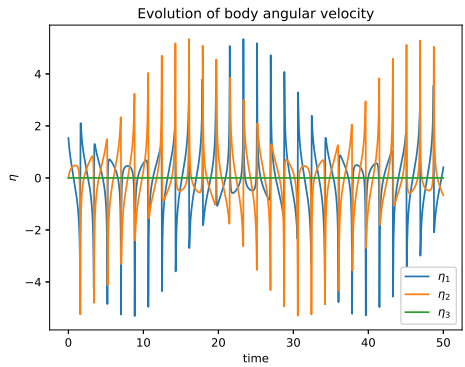}
		\caption{Evolution of the components of the velocity $\eta$.}
	\end{subfigure}
	\begin{subfigure}[b]{0.4\textwidth}
	\centering
		\includegraphics[width=7cm]{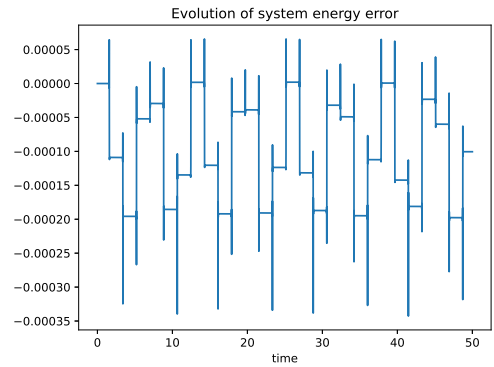}
		\caption{Evolution of the energy error.}
	\end{subfigure}
	\caption{Various evolution plots of the spherical Kepler problem computed with the 4-stage Lobatto method.}
	\label{fig:kepler_figures}
\end{figure}

In Fig.\ref{fig:kepler_orbits} and \ref{fig:kepler_figures} we offer several plots of numerical solutions of the problem for integrators of order 2, 4 and 6 corresponding to the 2, 3 and 4-stage Lobatto methods using the exponential map as retraction map $\tau$. The initial conditions have been chosen as $g_0 = (0.940125174120388, -0.693184358892293, 3.007331043590061)$, $\eta_0 = (1.534184084268850,0,0)$ and $\lambda_0 = \Lambda^1_0 = 0$.\\

This problem is fairly similar to the pendulum problem. As such it exhibits the same issues as that one, in particular, related to the 3-stage method.\\

\section{Conclusions and future work}
In this paper, high-order numerical integrators applied to homogeneous spaces have been presented as an application of nonholonomic partitioned RKMK methods on Lie groups. These methods are readily usable for any Lagrangian system on a homogeneous space. They offer flexibility in terms of how to exploit the symmetries of the system, allowing the choice of any trivialization, and in terms of how we decide to handle the constraints imposed by isotropy. The integrators obtained have provable order and seem to preserve in many cases the excellent behaviour of their variational counterparts. We have shown an application to a particular homogeneous space, $S^2 \cong SO(3)/SO(2)$, and two different dynamical systems, the spherical pendulum and the spherical Kepler problem.\\

However this algorithm has several drawbacks. One is that it requires a regular Lagrangian in the Lie group and there is currently no way to side-step this issue. Since by pulling the Lagrangian from the homogeneous space to the Lie group leads inevitably to a singular Lagrangian, this means the user is required to supply a suitable regularization. Fortunately, this can be easily accomplished provided the original Lagrangian is itself regular. Still, managing to eliminate this requirement would be desirable and a possible avenue of research.\\

A more complex matter that was mentioned in several occasions, is that these methods are affected by an instability issue most likely rooted in the way we handle the Lagrange multipliers. We believe that these instabilities can be fixed by supplying the right additional equations at each step. Finding such equations from first principles and geometry that can be applied universally is a matter of current research. The user may find some ad-hoc equations for the particular case and RK method used.

\bibliography{references}
\bibliographystyle{plain}

\end{document}